\documentclass[reqno, 12pt]{amsart}
\usepackage{a4wide, amsfonts, amssymb, amsmath, amsthm, mathrsfs, enumerate, enumitem, setspace, url}
\usepackage[utf8]{inputenc}
\usepackage{tikz-cd} 
	\usetikzlibrary{cd}
\usepackage{color}
\usepackage{hyperref} 
	\definecolor{darkred}{rgb}{0.5,0,0}
	\definecolor{darkgreen}{rgb}{0,0.5,0}
	\definecolor{darkblue}{rgb}{0,0,0.5}
	\hypersetup{linkcolor=darkblue,filecolor=darkgreen,urlcolor=darkred,citecolor=darkblue}

\usepackage{comment}
\usepackage{float}
\usepackage[T2A,T1]{fontenc}
\DeclareSymbolFont{cyrillic}{T2A}{cmr}{m}{n}
\DeclareMathSymbol{\Sha}{\mathalpha}{cyrillic}{216}
\usepackage[margin=1in]{geometry}

\usepackage[toc,page]{appendix}

\theoremstyle{plain}
\newtheorem{theorem}{Theorem}[section]
\newtheorem*{theorem*}{Theorem}
\newtheorem{proposition}[theorem]{Proposition}
\newtheorem{lemma}[theorem]{Lemma}
\newtheorem{corollary}[theorem]{Corollary}

\theoremstyle{remark}
\newtheorem{remark}[theorem]{Remark}
\newtheorem{example}[theorem]{Example}

\theoremstyle{definition}
\newtheorem{definition}[theorem]{Definition}
\newtheorem{notation}[theorem]{Notation}
\numberwithin{equation}{section}

\newcommand{\SESA}[4]{%
  \begin{tikzcd}[ampersand replacement=\&]
  0 \arrow[r] \& #1 \arrow[r] \& #2 \arrow[r,"#4"] \& #3 \arrow[r] \& 0 
  \end{tikzcd}%
}

\DeclareMathOperator{\HH}{H}
\DeclareMathOperator{\Gal}{Gal}
\DeclareMathOperator{\Pic}{Pic}
\DeclareMathOperator{\Br}{Br}
\DeclareMathOperator{\inv}{inv}

\DeclareMathOperator{\Res}{Res}
\DeclareMathOperator{\Spec}{Spec}

\DeclareMathOperator{\Frac}{Frac}

\DeclareMathOperator{\Norm}{Norm}

\DeclareMathOperator{\GL}{GL}
\DeclareMathOperator{\rank}{rank}

\DeclareMathOperator{\Cores}{Cores}

\DeclareMathOperator{\Tr}{Tr}
\DeclareMathOperator{\im}{Im}
\DeclareMathOperator{\et}{\acute{e}t}
\renewcommand{\epsilon}{\varepsilon}

\begin{document}
\onehalfspacing
\title[Integral points on symmetric affine cubic surfaces]{Integral points on symmetric affine cubic surfaces}
\thanks{2020 {\em Mathematics Subject Classification} 
     14G12 (primary), 11D25, 14G05, 14F22 (secondary).
}

\author{H.Uppal}
	\address{H.Uppal, Department of Mathematical Sciences, University of Bath, Claverton Down, Bath, BA2 7AY, UK}
	\email{hsu20@bath.ac.uk}

\begin{abstract}
We show that if $f(u)\in \mathbb{Z}[u]$ is a monic cubic polynomial, then for all but finitely many $n\in \mathbb{Z}$ the affine cubic surface
$f(u_{1})+f(u_{2})+f(u_{3})=n \subset \mathbb{A}^{3}_{\mathbb{Z}}$ has no integral Brauer-Manin obstruction to the Hasse principle.
\end{abstract}  

\maketitle
\setcounter{tocdepth}{1}
\tableofcontents
\section{Introduction}
Many interesting Diophantine problems involve finding solutions to inhomogeneous equations, for example a conjecture by Heath-Brown \cite{HB92} states that if $n$ is an integer that is not congruent to 4 or 5 modulo 9, then there exists infinity many integral solutions to the equation \begin{equation}\label{SumOfThreeCubesEquation}
    u_{1}^3+u_{2}^3+u_{3}^3 = n.
\end{equation}
The existence of integral solutions to (\ref{SumOfThreeCubesEquation}) was first studied by Mordell in \cite{M53}, later Miller and Woollett \cite{MW55} showed the existence of integral solutions for $0\leq n \leq 100$, with $n$ not congruent to 4 or 5 mod 9. A system of inhomogeneous equations with integral coefficients defines a scheme $\mathcal{U}$ over $\mathbb{Z}$ and we can view integral solutions to such a system as integral points on $\mathcal{U}$. There is a necessary condition for the existence of an integral point on $\mathcal{U}$, specifically denote by $U$ the base change of $\mathcal{U}$ to $\mathbb{Q}$, then $\mathcal{U}(\mathbb{Z})\neq\emptyset$ implies \begin{align*}
    \mathcal{U}(\mathbb{A}_{\mathbb{Z}}):= U(\mathbb{R}) \times \prod\limits_{p \ \text{prime}}\mathcal{U}(\mathbb{Z}_{p}) \neq \emptyset.
\end{align*}A natural question to ask is: when does $\mathcal{U}(\mathbb{A}_{\mathbb{Z}})\neq \emptyset$ imply $\mathcal{U}(\mathbb{Z})\neq \emptyset$? This implication is called the integral Hasse principle, but does not hold in general. Some failures of the integral Hasse principle can be explained by the integral Brauer-Manin obstruction, for example \cite[\S 8]{CTX09}, \cite{KT08} and \cite[Example 7.8]{JS17}. In this paper we will study families of log $K3$ surfaces over $\mathbb{Z}$, an example of such a family of log $K3$ surfaces would be the schemes over $\mathbb{Z}$ defined by (\ref{SumOfThreeCubesEquation}), parameterised by $n$. Colliot-Thélène and Wittenberg \cite[Thm 4.1a]{CTW12} showed that (\ref{SumOfThreeCubesEquation}) has no integral Brauer-Manin obstruction for any $n \in \mathbb{Z}$. We ask if a more general statement holds, specifically: given a monic cubic polynomial $f(u) \in \mathbb{Z}[u]$ and an integer $n$, can the affine surface
\begin{align*}
    \mathcal{U}_{n}: f(u_{1})+f(u_{2})+f(u_{3})=n
\end{align*} have an integral Brauer-Manin obstruction? In this paper we prove the following result, which extends our understanding of such questions.
\begin{theorem}\label{BIGTHM}
Let $f(u) \in \mathbb{Z}[u]$ be a monic cubic polynomial. Then for all but finitely many $n\in \mathbb{Z}$, the affine surface
\begin{align}\label{AffineCubicIntro}
    \mathcal{U}_{n}: f(u_{1})+f(u_{2})+f(u_{3})=n \subset \mathbb{A}^{3}_{\mathbb{Z}}
\end{align} has no integral Brauer-Manin obstruction.
\end{theorem} 
The structure of the paper is as follows. Sections \ref{ConicBundelSec} and \ref{SmoothCubicSurfaces} explain some basic facts about conic bundles and smooth cubic surfaces, respectively. In Section \ref{GeometryOfCompact} we calculate the Brauer group of the compactification $X_{n}$ of $U_{n}:=\mathcal{U}_{n}\times_{\mathbb{Z}}\mathbb{Q}$, if $X_{n}$ is smooth. In Section \ref{FinitenessResults} we show that under certain conditions on the coefficients of the defining equation for $X_{n}$, the Brauer group of $X_{n}$ is trivial for all but finitely many $n \in \mathbb{Z}$. In Section \ref{GeometryOfAffine} we calculate $\Br U_{n}$ under mild assumptions. In Section \ref{MainThms} we prove Theorem \ref{BIGTHM}. In Sections \ref{Examples} and \ref{SectionSA} we show explicit examples where our work can be applied, specifically we show that the affine surface defined by the sum of three tetrahedral numbers has no integral Brauer-Manin obstruction to the Hasse principle and that the surfaces $U_{n}$ can fail weak and strong approximation. All computations were done on SageMath \cite{sagemath} and Magma \cite{MAGMA}.  
\subsection{Acknowledgments}
The author would like thank Daniel Loughran for his insights and tireless support and J\"org Jahnel for donating his Magma code for computing the Brauer group of smooth cubic surfaces. They would also like to thank Sam Streeter and Olivier Wittenberg for comments on early revisions of this paper, J.-L. Colliot-Thélène for pointing out the sum of three cubes conjecture was first presented as a question of Mordell in the literature and the anonymous referee for many valuable comments and corrections.
\section{Preliminaries}
We fix some notation and definitions for the rest of the paper. 
\begin{notation}
Let $S$ be a scheme over a ring $R$. Given a morphism of rings $R\rightarrow T$ denote by $S\times_{R}T$ the base change $S\times_{\Spec(R)}\Spec(T)$, when the base ring $R$ is clear we drop this in our notation. The field $k$ will always be perfect and we denote by $\bar{k}$ the algebraic closure of $k$. If $S$ is a scheme over a field $k$, then denote by $\bar{S}$ the base change of $S$ to the algebraic closure of $k$. 
\end{notation}
\begin{definition}[Grothendieck-Brauer group]
Let $S$ be a scheme over a field $k$. The cohomological Brauer group of $S$ is the second étale cohomology group
$$
\Br S:=\HH^{2}_{\et}(S,\mathbb{G}_{m}).
$$
We define the filtration $$0 \subseteq \Br_{0} S \subseteq \Br_{1} S \subseteq \Br S$$
where $\Br_{0} S:=\im\left(\Br k \rightarrow \Br S\right)$ and $\Br_{1} S:=\ker(\Br S \rightarrow \Br \bar{S})$. We call $\Br_{1} S$ the \emph{algebraic Brauer group} and $\Br S/\Br_{1} S$ the \emph{transcendental Brauer group}.
\end{definition} 
\subsection{Integral Brauer-Manin obstruction}\label{IBMO} 
In this section we recall the integral Brauer-Manin set up, as introduced by Colliot-Thélène and Xu \cite{CTX09}.
Let $\Omega$ be the set of places of $\mathbb{Q}$ and denote by $\mathbb{Q}_{v}$ the completion of $\mathbb{Q}$ at $v \in \Omega$. For all places $v \in \Omega$ there exists an injective invariant map
$\inv_{v}:\Br \mathbb{Q}_{v} \rightarrow \mathbb{Q}/\mathbb{Z}$ which has image $\{0,1/2\}$ if $v$ is the real place and is an isomorphism if $v$ is finite. Furthermore, there is an exact sequence
\begin{equation*}
    \SESA{\Br \mathbb{Q}}{\bigoplus\limits_{v \in \Omega}\Br \mathbb{Q}_{v}}{\mathbb{Q}/\mathbb{Z}}{\sum\limits_{v\in \Omega}\text{inv}_{v}}.
\end{equation*}
Let $\mathcal{U}$ be a scheme of finite type over $\mathbb{Z}$ and let $U$ denote the base change of $\mathcal{U}$ to $\mathbb{Q}$. By \cite[Prop 13.3.1]{CTS21} for every $\alpha \in \Br U$ there exists a finite set of places $S_{\alpha}$ including all archimedean ones, such that the invariant map 
\[
\inv_{v}\alpha:\mathcal{U}(\mathbb{Z}_{v})\rightarrow \mathbb{Q}/\mathbb{Z}
\] is zero for all $v\not\in S_{\alpha}$. Hence, there exists a well defined pairing
$$
    \mathcal{U}(\mathbb{A}_{\mathbb{Z}}) \times \Br U \rightarrow \mathbb{Q}/\mathbb{Z}, \ \left((x_{v})_{v},\alpha\right)\mapsto \sum\limits_{v\in \Omega} \inv_{v}(\alpha(x_{v})).
$$ The following set 
$$\mathcal{U}(\mathbb{A}_{\mathbb{Z}})^{\Br}:=\left\{(x_{v})_{v}\in U(\mathbb{R}) \times \prod\limits_{p \ \text{prime}}\mathcal{U}(\mathbb{Z}_{p}):\sum\limits_{v \in \Omega}\inv_{v}(\alpha(x_{v}))=0\ \forall \alpha \in \Br U\right\}$$
is the \emph{integral Brauer-Manin set} and there are inclusions
$$
    \mathcal{U}(\mathbb{Z})\subseteq \mathcal{U}(\mathbb{A}_{\mathbb{Z}})^{\Br}\subseteq \mathcal{U}(\mathbb{A}_{\mathbb{Z}}).
$$
We say there is an \emph{integral Brauer-Manin obstruction to the Hasse principle} when 
$$
    \mathcal{U}(\mathbb{A}_{\mathbb{Z}}) \neq \emptyset \ \  \text{but}\ \  \mathcal{U}(\mathbb{A}_{\mathbb{Z}})^{\Br}= \emptyset.
$$

\section{Conic bundles}\label{ConicBundelSec}
In this section we explain some basic properties about conic bundles and give their relation to smooth cubic surfaces. For the rest of Section \ref{ConicBundelSec} we assume $k$ is a field of characteristic not equal to 2 or 3.

\begin{definition}
A \emph{conic bundle} over $k$ is a smooth projective surface $X$ over $k$ together with a dominant morphism $\pi:X \rightarrow \mathbb{P}^{1}_{k}$, all of whose fibres are isomorphic to either a smooth conic or a union of two lines. 
\end{definition}

\begin{remark}
The Picard group of a conic bundle $X$ is free and finitely generated. As $\Pic X$ is a subgroup of $\Pic \bar{X}$ it is sufficient to show that $\Pic \bar{X}$ is free and finitely generated. Any conic bundle over an algebraically closed field is rational (over $\mathbb{P}^{1}$) \cite[\S 2, Prop 1b]{I79}, using \cite[\S 9, Thm 2.2]{Q02} and \cite[Ex 8.5a]{H77} it follows that $\Pic\bar{X}$ is free and finitely generated.
\end{remark}
\begin{definition}
A conic $C$ over $k$ is called \emph{split} if it is either smooth or isomorphic to a union of two rational lines over $k$.
\end{definition}
\begin{definition}
A smooth projective surface $X$ over $k$ is said to be \emph{split} if the morphism $\Pic X\rightarrow \Pic \bar{X}$ is an isomorphism.
\end{definition}
\begin{lemma}\label{RepLemma}
Let $G$ be a finite group acting on a finite dimensional vector space $V$ over $\mathbb{R}$. Denote by $V^{G}$ the vectors of $V$ which are invariant under the action of $G$. If the action of $G$ is non-trivial then $\dim_{\mathbb{R}} V^{G}<\dim_{\mathbb{R}} V$.
\end{lemma}
\begin{proof}
The action of $G$ on $V$ defines a representation $\rho:G\rightarrow \GL(V)$, and as $V$ is a finite dimensional vector space. We have that $V^{G}$ is a vector subspace of $V$, hence if $\dim V^{G}=\dim V$, this implies that $\rho$ is the trivial representation i.e.\ $G$ acts trivially on $V$.
\end{proof}

\begin{lemma}\label{EnoughLines}
Let $\pi:X\rightarrow \mathbb{P}^{1}_{k}$ be a conic bundle surface over $k$, such that $X(k)\neq \emptyset$. If every singular fibre of $\pi$ over $\overline{k}$ is defined over $k$ and split, then $X$ is split over $k$.
\end{lemma}
\begin{proof} Let $L \supseteq k$ be a field extension. Using \cite[Lemma 2.1]{FLS18} we have that for $X_{L}:=X\times_{k}L$ \begin{equation}\label{RankSingFib}
    \rank \Pic \left(X_{L}\right)=2+\#\{\text{closed point }p \in \mathbb{P}^{1}_{L}\  :  \pi^{-1}(p)\  \text{singular and split over}\ \kappa(p)\}
\end{equation} where $\kappa(p)$ is the residue field of the point $p$. Under the assumption that each singular fibre is defined and split over $k$, by (\ref{RankSingFib}) $\rank \Pic X=\rank \Pic\bar{X}$. Moreover, as $X$ is a projective scheme over a field $k$ with $X(k) \neq \emptyset$ by \cite[Rem 5.4.3]{CTS21}
$
    \Pic X=\left(\Pic \bar{X}\right)^{\Gal(\bar{k}/k)}.
$ By definition $\dim_{\mathbb{R}}\left(\Pic(X)\otimes_{\mathbb{Z}} \mathbb{R}\right)=\rank \Pic X$. The absolute Galois group acts on $\Pic \bar{X}$ via a finite subgroup $G\subset \Gal(\bar{k}/k)$, as $\dim_{\mathbb{R}}\left(\Pic(X)\otimes_{\mathbb{Z}} \mathbb{R}\right)=\dim_{\mathbb{R}}\left(\Pic (\bar{X})\otimes_{\mathbb{Z}} \mathbb{R}\right)$ by Lemma \ref{RepLemma} the action of $\Gal(\bar{k}/k)$ on $\Pic(\bar{X})\otimes_{\mathbb{Z}} \mathbb{R}$ is trivial, hence $\Pic X=\Pic \bar{X}$. 
\end{proof}
\begin{lemma}\label{SplitFibres}
Let $\pi: X\rightarrow \mathbb{P}^{1}_{k}$ be a conic bundle over $k$. If $\pi$ admits a section over $k$, then every singular fibre of $\pi$ is split over $k$.
\end{lemma}
\begin{proof}Let $s$ be a section of $\pi$. As $s$ is a well-defined map it will meet each fibre exactly once. Moreover, as $X$ is regular, any section will intersect each fibre at a smooth point \cite[\S 9 Cor 1.32]{Q02}. As the section $s$ can only meet singular fibres at smooth points, i.e. it intersects only one of the lines in the singular fibre, this line is Galois-invariant and the singular fibre is split over $k$.
\end{proof}
\begin{remark}\label{ConstructConicBundle}
If $X$ is a smooth cubic surface defined over $k$ with a line $L \subset X$, then $X$ has a conic bundle structure. Up to a change of variables we can assume $L:=\{x_{1}=0,x_{3}=0\}$. As $X$ is a smooth cubic surface we can write the defining equation for $X$ as $x_{1}C_{1}+x_{3}C_{2}$ where each $C_{i}$ is a form of degree 2 in variables $x_{0},x_{1},x_{2},x_{3}$. Consider the pencil of planes which contain $L$ 
$$
    P_{s,t}:sx_{1}-tx_{3}=0.
$$ Then the residual intersection of $P_{s,t}$ with $X$ is a conic which we denote by $C_{s,t}$. As $X$ is smooth $C_{1}$ and $C_{2}$ have no common intersection on $L$, hence at a point $Q:=[x_{0}:x_{1}:x_{2}:x_{3}]$ if $(x_{1},x_{3})=(0,0)$ then $(C_{1},C_{2})\neq (0,0)$. We can then define a morphism $\pi:X\rightarrow \mathbb{P}^{1}_{k}$ by
\begin{equation*}
  Q:=[x_{0},x_{1},x_{2},x_{3}] \mapsto
    \begin{cases}
      [x_{1}:x_{3}] \ &\text{if} \ (x_{1},x_{3}) \neq (0,0), \\
      [-C_{2}(Q):C_{1}(Q)] \ &\text{if} \ (C_{1}(Q),C_{2}(Q)) \neq (0,0),
    \end{cases}       
\end{equation*} i.e.\ $\pi$ is the projection away from the line $L$. Moreover, this is a dominant morphism and the fibre above each closed point of $\mathbb{P}^{1}_{k}$ is a conic, hence we have a conic bundle structure. If we consider the geometric point $[t:s] \in \mathbb{P}^{1}_{\bar{k}}$, the fibre above this point is the conic $C_{s,t}$. Consider the representative matrix of the generic fibre of $\pi$, which we denote by $T$. Denote by $\Delta(s,t)$ the determinant of $T$; this is a function in $s$ and $t$. Then the geometric points $[\beta:\alpha]$ of $\mathbb{P}^{1}_{\bar{k}}$ which have a singular fibre are exactly those that satisfy $\Delta(\alpha,\beta)=0$. We call $\pi$ the \emph{conic bundle map of $X$ associated to $L$}.
\end{remark}
\section{Smooth cubic surfaces}\label{SmoothCubicSurfaces}
We describe how to calculate the first Galois cohomology group for $\Pic(\bar{X})$ where $X$ is a smooth cubic surface over a field $k$ of characteristic not equal to 2 and 3. In the specific case where $X$ is defined over a number field $L$ and everywhere locally soluble, by the Hochschild-Serre spectral sequence and the fact $\HH^{3}(L,L^{*})=0$ we get an exact sequence
\[
    0\rightarrow \Pic X \rightarrow \HH^{0}(L,\Pic \bar{X})\rightarrow \Br_{0} X\rightarrow \Br_{1} X\rightarrow \HH^{1}(L,\Pic\bar{X})\rightarrow 0
\] which gives an isomorphism 
\[
    \Br_{1}X/\Br_{0}X=\Br X/\Br L \cong \HH^{1}(L,\text{Pic}\bar{X}).
\]
\begin{definition}
Let $X$ be a smooth cubic surface defined over $k$. The \emph{splitting field} $K$ of $X$ is the smallest normal extension of $k$ for which all the 27 lines on $\bar{X}$ are defined. Equivalently, $K$ is the smallest normal extension of $k$ such that $X\times_{k}K$ is split.
\end{definition}
\begin{remark}\label{GaloisActionRem}
The action of $\Gal(\bar{k}/k)$ on $\bar{X}$ permutes the 27 lines on $\bar{X}$, which induces an automorphism of $\Pic \bar{X}$, preserving intersection pairings and the anticanonical divisor. The group of all automorphisms of $\text{Pic} \bar{X}$ which preserve intersection pairings and the anticanonical divisor is the Weyl group of the $E_{6}$ root system, denoted by $W(E_{6})$ \cite[\S IV, 25]{Man86}. The Galois action on $\bar{X}$ gives rise to a homomorphism $\rho:\Gal(\bar{k}/k)\rightarrow W(E_{6})$. The Galois group $\Gal(\bar{k}/k)$ acts upon the lines of $\bar{X}$ via a finite group $G:=\text{Im}(\rho)$, and if $K$ is the splitting field of $X$, there is an isomorphism $\text{Im}(\rho) \cong \Gal(K/k)$. Once one knows the Galois action on the lines of $\bar{X}$, one can use Magma to compute $\HH^{1}(k,\Pic \bar{X})$.
\end{remark} 
\section{Geometry of the compactification}\label{GeometryOfCompact}
 When calculating the integral Brauer-Manin obstruction for the affine surfaces (\ref{AffineCubicIntro}) we need to calculate the Brauer group of $U_{n}:=\mathcal{U}_{n}\times_{\mathbb{Z}} \mathbb{Q}$. If the compactification $X_{n}$ of $U_{n}$ is smooth, by Grothedieck's purity theorem \cite[Thm 3.7.1]{CTS21} there is an inclusion $\Br X_{n} \hookrightarrow \Br U_{n}$. Therefore, the first step in calculating $\Br U_{n}$ is to calculate $\Br X_{n}$. Throughout Section \ref{GeometryOfCompact} fix a field $k$ of characteristic not equal to 2 or 3 and a polynomial $f(u)=u^3+a_{2}u^2+a_{1}u^2+a_{0}\in k[u]$. For $n\in k$ denote by $U_{n}$ the affine cubic surface $$
 U_{n}: f(u_{1})+f(u_{2})+f(u_{3})=n \subset \mathbb{A}_{k}^{3}.
 $$ By completing the cube and making a change of variables we can assume that $$
U_{n}: F(1,u_{1})+F(1,u_{2})+F(1,u_{3})=n\subset \mathbb{A}^{3}_{k}
$$ where $F(x_{0},x_{i})=x_{i}^3+ax_{i}x_{0}^2+bx_{0}^3\in k[x_{0},x_{i}]$. The compactification of $U_{n}$ is $$
X_{n}:F(x_{0},x_{1})+F(x_{0},x_{2})+F(x_{0},x_{3})=nx_{0}^3\subset \mathbb{P}^{3}_{k}
$$ where $u_{i}=x_{i}/x_{0}$ for $i=1,2,3$. We require $X_{n}$ to be smooth to apply Grothedieck's purity theorem; in Section \ref{SmoothnessSection1} we show that outside of a finite set of $n\in k$, the cubic surface $X_{n}$ is always smooth. This coincides with the aim of Theorem \ref{BIGTHM}, as we want to show that outside a finite set of $n \in \mathbb{Z}$, the affine cubic surface (\ref{AffineCubicIntro}) has no integral Brauer-Manin obstruction, therefore we can restrict ourselves to the cases where $X_{n}$ is smooth for any future calculations. 
\subsection{Smooth compactifications}\label{SmoothnessSection1} 
\begin{lemma}\label{SmootnessLemma}Define
\begin{equation*}
    X: F(x_{0},x_{1})+F(x_{0},x_{2})+F(x_{0},x_{3})-nx_{0}^3 \subset \mathbb{P}^{3}_{k}\times \mathbb{A}^{1}_{k}
\end{equation*} where $\mathbb{A}^{1}_{k}$ has the coordinate $n$. Then the projection map $\pi: X \rightarrow \mathbb{A}^{1}_{k}$ has a smooth generic fibre.
\end{lemma}
\begin{proof} The generic fibre is  $X_{\eta}:=X_{n} \times_{k} k(n) \subset \mathbb{P}^{3}_{k(n)}$. Let
$$
 G(x_{0},x_{1},x_{2},x_{3}) := x_{1}^3+x_{2}^3+x_{3}^3+a(x_{1}+x_{2}+x_{3})x_{0}^2+(3b-n)x_{0}^3,
$$
then if there exists a non-smooth point $p=[x_{0}:x_{1}:x_{2}:x_{3}]$ lying on $X_{\eta}$, by applying the Jacobian criterion for smoothness we see that the partial derivatives of $G$ will vanish when evaluated at $p$. Explicitly, $\frac{\partial G}{\partial x_{i}}(p) = 3x_{i}^2+ax_{0}^2=0$ for $i=1,2,3$ and $\frac{\partial G}{\partial x_{0}}(p) =2a(x_{1}+x_{2}+x_{3})x_{0}+3(3b-n)x_{0}^2=0$. As $\frac{\partial G}{\partial x_{0}}(p)=0$ this implies that 
$
(2a(x_{1}+x_{2}+x_{3})+3(3b-n)x_{0})x_{0}=0,
$ hence either $x_{0}=0$ or $2a(x_{1}+x_{2}+x_{3})+3(3b-n)x_{0}=0$. If $x_{0}=0$, this implies $x_{1}=x_{2}=x_{3}=0$, hence such a non-smooth point cannot exist. We can now assume that $x_{0}=1$, using the partial derivative $\frac{\partial G}{\partial x_{i}}$ we see $x_{i}=\pm\sqrt{-a/3}$ for $i=1,2,3$. If $x_{1}=x_{2}=x_{3}=\sqrt{-a/3}$, then \[
G(1,\sqrt{-a/3},\sqrt{-a/3},\sqrt{-a/3})=3(\sqrt{-a/3})^3+3a\sqrt{-a/3}+(3b-n)=0,
\] hence $n = 3b-\frac{2}{3} \, \sqrt{3} (-a)^{3/2} $. As $3b-\frac{2}{3} \, \sqrt{3} (-a)^{3/2}\in \bar{k}$ but $n\not \in \bar{k}$, this is a contradiction. The other possibilities for $x_{1},x_{2},x_{3}$ lead to the same contradiction, hence we see that the generic fibre $X_{\eta}$ is smooth over $k(n)$.
\end{proof}
\begin{lemma}\label{lemma: not seperable implies singular} If the cubic polynomial $f_{1}(x)=x^3+ax+3b-n\in k[x]$ is not separable then the surface $X_{n}$ is singular.
\end{lemma}
\begin{proof}
$f_{1}$ is not separable if and only if the discriminant of $f_{1}$ is zero i.e. $4a^3+27(3b-n)^2=0$. Then the surface $X_{n}$ has a singular point $[x_{0}:x_{1}:x_{2}:x_{3}]=[1:\sqrt{-a/3}:\sqrt{-a/3}:-\sqrt{-a/3}]$.
\end{proof}
\begin{proposition}\label{SmoothFinite} There are only finitely many choices of $n \in k$ such that $X_{n}$ is not smooth.
\end{proposition}
\begin{proof} We saw in Lemma \ref{SmootnessLemma} that the generic fibre of the projection map $\pi: X \rightarrow \mathbb{A}^{1}_{k}$ is smooth. This means that the smooth locus of the morphism $\pi$ is an open subset of $\mathbb{A}^{1}_{k}$ which contains the generic point, thus there can only be finitely many singular fibres.
\end{proof}

For the rest of Section \ref{GeometryOfCompact}, fix a choice of $n\in k$ is such that $X_{n}$ is smooth.
\subsection{Conic bundle structure on $X_{n}$}
We can now explicitly describe the conic bundle structure on $X_{n}$ over a possible finite extension $l$ of $k$. Let
$f_{1}(x)=x^3+ax+3b-n$. If $f_{1}(x)$ is irreducible, let $l:=k[x]/(f_{1}(x))$ and let $r$ be the image of $x$ in $l$. Otherwise, let $l:=k$ and let $r$ be a root of $f_{1}(x)$ in $l$. Then $X_{n}\times_{k}l$ contains a line $L$, where $$
L:x_{2}+x_{3}=0,\  x_{1}-rx_{0}=0.
$$ We then run the construction given in Remark \ref{ConstructConicBundle} and this defines the conic bundle map $\pi$ of $X_{n}$ associated to the line $L$. The points of $\mathbb{P}^{1}_{l}$ over which $\pi$ has singular fibres are \begin{equation*}
    p = \begin{cases}
      (t),\\
      (s+t),\\
      \left(at^{3}+3r^2st^{2}-3r^2s^{2}t +(4a+3r^2)s^{3} \right).
    \end{cases}
\end{equation*} 
The singular fibres of $\pi$ over $\bar{k}$ are defined over the splitting field of $ax^3+3r^2x^{2}-3r^2x+(4a+3r^2)$ over $l$.
\begin{remark}\label{rem: label lines}
To fix some notation for the rest of the paper, we label nine lines on $\bar{X}_{n}$. Let $r_{1},r_{2},r_{3}$ be roots of $f_{1}(x)$ in $\bar{k}$. The surface $\bar{X}_{n}$ contains the lines\begin{align*}
L_{1,1}&:x_{2}+x_{3}=x_{1}-r_{1}x_{0}=0, \ L_{1,2}:x_{2}+x_{3}=x_{1}-r_{2}x_{0}=0,\ L_{1,3}:x_{2}+x_{3}=x_{1}-r_{3}x_{0}=0,\\
    L_{2,1}&:x_{1}+x_{3}=x_{2}-r_{1}x_{0}=0,\ L_{2,2}:x_{1}+x_{3}=x_{2}-r_{2}x_{0}=0, \ L_{2,3}:x_{1}+x_{3}=x_{2}-r_{3}x_{0}=0,\\
    L_{3,1}&:x_{1}+x_{2}=x_{2}-r_{1}x_{0}=0, \ L_{3,2}:x_{1}+x_{2}=x_{3}-r_{2}x_{0}=0,\ L_{3,3}:x_{1}+x_{2}=x_{3}-r_{3}x_{0}=0.
\end{align*} Note that each row of lines define three coplanar lines, meeting at an Eckardt point.
\end{remark}
\subsection{Splitting field}\label{SplittingFieldSection}
\begin{lemma}\label{ExistsSection}
Denote by $L$ the splitting field of $f_{1}$ over $k$. The conic bundle map associated to $L_{1,1}$ on $X_{n}$ admits a section over $L$, namely $L_{2,2}$.
\end{lemma}
\begin{proof} It is sufficient to show that $L_{2,2}$ does not intersect $L_{1,1}$. Let $P$ be the plane $x_{2}+x_{3}=0$, this plane contains $L_{1,j}$ for $j\in \{1,2,3\}$. The line $L_{2,2}$ intersects $P$ at the point $p=[1:r_{2}:r_{2}:-r_{2}]$. Moreover, if $p\in L_{1,1}$ this would imply that $r_{1}=r_{2}$. This is only possible if $f_{1}$ is not separable, by Lemma \ref{lemma: not seperable implies singular} this implies $X_{n}$ is not smooth contradicting our assumption on $X_{n}$.
\end{proof}
\begin{proposition}\label{SplittingFieldProp} $X_{n}$ is split over the splitting field of the two polynomials 
$$ f_{1}(x)=x^3+ax+3b-n\in k[x], \ g_{1}(x) =ax^3+3r^2(x^{2}-x+1)+4a \in k(r)[x]$$
where if $f_{1}$ is reducible over $k$, $r$ is a root of $f_{1}$; if $f_{1}$ is irreducible over $k$, $r$ is the image $x$ in the field extension $k[x]/(f_{1}(x))$.
\end{proposition}
\begin{proof}
Let $K$ be the splitting field of $f_{1}$ and $g_{1}$ over $k$ and $\pi:X_{n} \rightarrow \mathbb{P}^{1}_{K}$ be the conic bundle map associated to $L_{1,1}$ on $X_{n}$. All singular fibres over $\bar{k}$ are defined over $K$ and by Lemma \ref{ExistsSection} there exists a section of $\pi$, hence by Lemma \ref{SplitFibres} every fibre of $\pi$ is split. By Lemma \ref{EnoughLines} $X_{n}$ is split over $K$.
\end{proof}
Even though we have polynomials that define the splitting field for $X_{n}$, one of the polynomials is defined over $k$ and the other over a possible finite extension of $k$. To understand the Galois action on the lines of $X_{n}$ it will be easier if both polynomials are always defined over $k$; the rest of the results in Section \ref{SplittingFieldSection} are devoted to obtaining two such polynomials over $k$ for all but finitely many choices of $n\in \mathbb{Z}$. In the case $a=0$ we see that $f_{1}(x)=x^3+3b-n$ and $g_{1}(x)=3r^2(x^2-x+1)$ and as the splitting field of $f_{1}$ will contain a third root of unity $\zeta_{3}$ we have that $x^2-x+1$ will factor as $(x+\zeta_{3})(x+\zeta_{3}^2)$, hence the splitting field of $f_{1}$ and $g_{1}$ is the same as the splitting field of $f_{1}$. For the rest of Section \ref{SplittingFieldSection} we assume $a$ and $4a^3+27(3b-n)^2$ are non-zero. Let $l:=k(r)=k[x]/(f_{1}(x))$ if $f_{1}$ is irreducible or $l:=k$ where $r$ is a root of $f_{1}$ if $f_{1}$ is reducible.
\begin{lemma}\label{KernelOneWay}
Denote by $\xi_{1}$ the $l$-algebra homomorphism $$
\xi_{1}: l[x]\rightarrow l[y],\  x\mapsto -a y^{2} - {\left(3 \, r^{2} - a\right)}y +(4a+3r^2)
$$ then the image of the ideal generated by $f_{2}(x)=x^3 - 12ax^2 + 36a^2x + 27(3b-n)^2 + 4a^3$ is contained in the ideal generated by $g_{1}(y)$.
\end{lemma}
\begin{proof}
The image of $f_{2}(x)$ under $\xi_{1}$ is $f_{2}(\xi_{1}(x))$. By applying Euclidean division to $f_{2}(\xi_{1}(x))$ and $g_{1}(y)$ we see that $f_{2}(\xi_{1}(x))=p(y)g_{1}(y)+q(y)$ where
\begin{align*}
    p(y)&=-a^{3} y^{3} + 18 \, a r^{4} - 6 \, a^{2} r^{2} + 5 \, a^{3} - 3 \, {\left(2 \, a^{2} r^{2} - a^{3}\right)} y^{2} - 3 \, {\left(3 \, a r^{4} - 5 \, a^{2} r^{2} + a^{3}\right)} y,\\
    q(y)&=-27 \, r^{6} - 54 \, a r^{4} - 27 \, a^{2} r^{2} + 243 \, b^{2} - 162 \, b n + 27 \, n^{2}.
\end{align*}
Factoring $q(y)$ we see that $q(y)=(-27 \, r^{3} - 27 \, a r + 81 \, b - 27 \, n)(r^{3} + a r + 3 \, b - n)$, hence $q(y)=0$ in $l[y]$ as $r$ is root of $f_{1}$. Thus $\xi_{1}((f_{2}))\subseteq (g_{1})$.
\end{proof}

\begin{lemma}\label{KernelOtherWay}
Denote by $\xi_{2}$ the $l$-algebra homomorphism $\xi_{2}: l[y]\rightarrow l[x]$ where $y$ is mapped to\[
 - \frac{{\left(6 \, a r^{2} + 4 \, a^{2} - 9 \, {\left(3 \, b - n\right)} r\right)} x^{2}-{\left(36 \, a^{2} r^{2} + 28 \, a^{3} - 54 \, a {\left(3 \, b - n\right)} r + 27 \, {\left(3 \, b - n\right)}^{2}\right)} x}{3 \, {\left(4 \, a^{3} + 27 \, {\left(3 \, b - n\right)}^{2}\right)} a} + \frac{r^{2}}{a}+\frac{4}{3}
\] then the image of the ideal generated by $g_{1}(y)$ is contained in the ideal generated by $f_{2}(x)$. 
\end{lemma}
\begin{proof}
This was done by a SageMath computation.
\end{proof}
\begin{lemma}\label{lemma: iso of cubic extensions} The $l$-algebra morphisms \[
\xi_{1}:l[x]/(f_{2}(x))\rightarrow l[y]/(g_{1}(y)) \text{ and }\xi_{2}:l[y]/(g_{1}(y))\rightarrow l[x]/(f_{2}(x))
\] are well defined and inverses of each other.
\end{lemma}
\begin{proof}
The fact that $\xi_{1}$ and $\xi_{2}$ are well defined follows from Lemmas \ref{KernelOneWay},\ref{KernelOtherWay}. To check they are inverses was done by a SageMath computation.
\end{proof}

\begin{proposition}\label{ProofOfkAlg2}
 Consider the cubic polynomials $$f_{2}(x) = x^3 - 12ax^2 + 36a^2x + 27(3b-n)^2 + 4a^3\  \text{and} \ \  g_{1}(x)=ax^3+3r^2(x^2-x+1)+4a.$$ defined over $l$. Then there exists a $l$-algebra isomorphism 
\[
l[x]/(f_{2}(x))\xrightarrow[]{\sim}l[y]/(g_{1}(y)).
\]
\end{proposition}
\begin{proof}
This follows from Lemma \ref{lemma: iso of cubic extensions}.
\end{proof}
\begin{proposition}\label{prop: changing g1 to f2}
  If $a$ and $4a^3+27(3b-n)^2$ are non-zero, then the compositum of the splitting fields of $f_{1}$ and $f_{2}$ is isomorphic to the compositum of the splitting fields of $f_{1}$ and $g_{1}$.
\end{proposition}
\begin{proof}
This follows immediately from Proposition \ref{ProofOfkAlg2}.
\end{proof}
\subsection{Brauer group of compactification} We have shown that the splitting field $K$ of $X_{n}$ is the compositum of the splitting fields of $f_{1}$ and $g_{1}$. We are now in a position to determine the Brauer group of $X_{n}$. In the case $a=0$ this is completely determined by Colliot-Thélène and Wittenberg \cite[Prop 2.1]{CTW12}. We determine the Brauer group in the cases $a\neq 0$.

\begin{remark}[{\cite{KC10}},Specialisations I] In this remark we explain that if \[
f:=f(t_{1},...,t_{m},x)\in \mathbb{Q}(t_{1},...,t_{m})[x]
\] is irreducible and $t_{i}$ are intermediates, given $n\in \mathbb{A}^{m}(\mathbb{Q})$ such that $n$ is not a pole of any of the coefficients of $f$ and $g(x):=f(n,x)\in \mathbb{Q}[x]$ is separable then there is an embedding $\Gal(g)\hookrightarrow \Gal(f)$. It is sufficient to show this for $m=1$. Let $A=\mathbb{Q}[t]$, $F=\Frac(A)$ and $f(t,x)\in \mathbb{Q}[t,x]$ which is irreducible. Furthermore, let $F'$ be $F$ adjoining a $x$-root of $f(t,x)$ and $A'$ the integral closure of $A$ in $F'$. If a specialisation $t=t_{0}\in \mathbb{Q}$ is such that $f(t_{0},x)$ is separable in $\mathbb{Q}[t]$ then the prime ideal $(t-t_{0})\subset A$ is unramified in $A'$. Let $E$ be the Galois closure of $F'/F$ and $B$ the integral closure of $A$ in $E$. Then as $(t-t_{0})$ is unramified in $A'$ it is unramified in $B$. Pick a prime $\mathfrak{B}\subset B$ above $(t-t_{0})$ then $(B/\mathfrak{B})/(A/(t- t_{0}))$ is a finite extension. It is clear $A/(t-t_{0}) =\mathbb{Q}$. Denote by $K$ the field $K:=B/\mathfrak{B}$. As $K/\mathbb{Q}$ is an extension of fields of characteristic 0 it is separable and as $E/F$ is Galois $K/\mathbb{Q}$ is normal, hence $K/\mathbb{Q}$ is also Galois. Furthermore, $K$ is the Galois closure of $f(t_{0},x)$ over $\mathbb{Q}$. As $\Gal(K/\mathbb{Q})$ is isomorphic to the  decomposition group $D_{\mathfrak{B}}$ there is an embedding $\Gal(K/\mathbb{Q})\hookrightarrow \Gal(E/F)$.
\end{remark}

\begin{proposition}\label{FunctionFieldGalois}Let $F(x_{0},x_{i})=x_{i}^3+ax_{i}x_{0}^2+bx_{0}^3\in \mathbb{Q}(a,b,n)[x_{0},x_{1}],$ and denote by $X_{a,b,n}$ the cubic surface $$
X_{a,b,n}:F(x_{0},x_{1})+F(x_{0},x_{2})+F(x_{0},x_{2})=nx_{0}^3
$$ over $\mathbb{Q}(a,b,n)$. Denote by $L$ the splitting field of $X_{a,b,n}$. Then $\Gal(L/\mathbb{Q}(a,b,n))\cong S_{3}\times S_{3}$.
\end{proposition}
\begin{proof}
   Note that if we specialise $a,b$ and $n$ to values in $\mathbb{Q}$ such that $X_{n}$ is smooth with splitting field $K$, then as abstract groups $\Gal(K/\mathbb{Q})$ is a subgroup of $\Gal(L/\mathbb{Q}(a,b,c,n))$. Choose $a=19, b = 8$ and $n=5$. Using Magma, we see that $X_{n}$ is smooth and has splitting field $K$ such that $\Gal(K/\mathbb{Q})\cong S_{3}\times S_{3}$, hence $[L:\mathbb{Q}(a,b,c,n)]\geq 36$. Furthermore, by Proposition \ref{SplittingFieldProp}, $[L:\mathbb{Q}(a,b,c,n)] \leq 36$, hence $[L:\mathbb{Q}(a,b,c,n)] = 36$. As $S_{3} \times S_{3} \leq \Gal(L/\mathbb{Q}(a,b,c,n))$ and the degree of the extension of $K$ is 36 we can deduce that $\Gal\left(L/\mathbb{Q}(a,b,c,n)\right)\cong S_{3}\times S_{3}$.
\end{proof}{}
\begin{remark}[Specialisations II]
There is an isomorphism $\Gal(L/\mathbb{Q}(a,b,n))\rightarrow N \subset W(E_{6})$ where $N$ is a subgroup of $W(E_{6})$. Moreover, $N$ has an action on the 27 exceptional vectors in $E_{6}\otimes \mathbb{Q}$ and this is the action of $\Gal(L/\mathbb{Q}(a,b,n))$ on the 27 lines of $X_{a,b,n}$. When we specialise to some values in $a,b,n$ such that $f_{1}$ and $g_{1}$ are separable and $X_{n}$ is a smooth cubic surface we have a homomorphism, $\Gal(K/\mathbb{Q})\hookrightarrow\Gal(L/\mathbb{Q}(a,b,n)) \rightarrow N$ where the image of $\Gal(K/\mathbb{Q})$ inside $N$ is some subgroup $G$ of $N$ and the Galois action on the 27 lines of $X_{n}$ is the action of $G$ on the 27 exceptional vectors of $E_{6}\otimes \mathbb{Q}$. Fix a choice of $a,b,n\in \mathbb{Q}$ such that $X_{n}$ is smooth and the compositum of the splitting fields of $f_{1}$ and $g_{1}$ generate a $S_{3}\times S_{3}$ extension of $\mathbb{Q}$. Then the sets $\{L_{1,i}:i=1,2,3\}$,$\{L_{2,i}:i=1,2,3\}$ and $\{L_{3,i}:i=1,2,3\}$ form Galois orbits of size three. Enumerating the subgroups of $W(E_{6})$ which have three orbits of size three and are isomorphic to $S_{3}\times S_{3}$ leads to two possibilities for the orbit types, namely $[3,3,3,18]$ and $[3,3,3,3,3,3,9]$. Moreover, using Magma we can show that for the case $a=19, b = 8$ and $n=5$ the Galois orbit type is $[3,3,3,18]$ hence, the Galois orbit type of $X_{a,b,n}$ is $[3,3,3,18]$. Now we know every possible Galois action on any smooth specialisation of $X_{a,b,n}$ such that $f_{1}$ and $g_{1}$ are separable. This is described fully in Table \ref{PossGalType}. We also use Magma to compute $\HH^{1}(\mathbb{Q},\Pic \bar{X})$ for any smooth cubic surface $X$ over $\mathbb{Q}$ with the Galois actions listed in Table \ref{PossGalType}.
\end{remark}
\begin{proposition}\label{TrivialBrauerGroup1} Choose $n\in \mathbb{Q}$ such that $X_{n}$ is smooth over $\mathbb{Q}$. Denote by $K$ the splitting field of $X_{n}$. If $\Gal(K/\mathbb{Q})\cong C_{2}\times S_{3}$ and $f_{1}$ is reducible or $\Gal(K/\mathbb{Q})\cong S_{3}\times S_{3}$ then $\Br X_{n}=\Br \mathbb{Q}$.
\end{proposition}
\begin{proof}
The cubic surface $X_{n}$ always has a rational point, namely $[x_{0}:x_{1}:x_{2}:x_{3}]=[0:1:-1:0],$ hence it is everywhere locally soluble. By the Hochschild-Serre spectral sequence we can conclude that $\Br X/\Br \mathbb{Q} \cong \HH^{1}(\mathbb{Q},\Pic \bar{X}_{n}),$ and then using the results from Table \ref{PossGalType} the proposition follows.
\end{proof}
\begin{notation}
Denote by $\Delta_{i}(n)$ the discriminant of $f_{i}$, for a fixed choice of $n\in \mathbb{Z}$. As $a,b,n\in \mathbb{Z}$, this implies that $\Delta_{i}(n)\in \mathbb{Z}$ for $i=1,2$. Furthermore, denote by $\Delta_{3}(n)$ the quotient $\Delta_{2}(n)/\Delta_{1}(n)$ if $4a^3+27(3b-n)^2\neq 0$ or $\Delta_{1}(n)/\Delta_{2}(n)=0$ otherwise.
\end{notation}
\begin{lemma}\label{BothPolyFullExt}
Let $$A_{1}:=\{n\in \mathbb{Z}:\Delta_{1}(n),\Delta_{2}(n)\ \text{or}\ \Delta_{3}(n)\ \text{is a square or either of}\ f_{1}\ \text{or}\ f_{2}\ \text{is reducible}\}$$ then for all $n\notin A_{1}$. Then the compositum of the splitting field of $f_{1}$ and $f_{2}$ is a $S_{3}\times S_{3}$-extension of $\mathbb{Q}$.
\end{lemma}
\begin{proof}
 By Proposition \ref{FunctionFieldGalois} the Galois group of the splitting field $K$ of $f_{1}$ and $f_{2}$ over $\mathbb{Q}$ is a subgroup of $S_{3}\times S_{3}$ so it is sufficient to show that $K$ has degree 36 for $n \notin A_{1}$. Choose $n\in \mathbb{Z}$ such that $n\not\in A_{1}$ then $\mathbb{Q}(r_{i})=\mathbb{Q}[x]/(f_{i})$ is a degree 3 extension of $\mathbb{Q}$. Moreover, as $\Delta_{3}(n)$ is not a square $\left[\mathbb{Q}(r_{1},r_{2}):\mathbb{Q}\right]=9$. As $\Delta_{i}(n)$ is not a square $\left[\mathbb{Q}\left(\sqrt{\Delta_{i}(n)}\right):\mathbb{Q}\right]=2$ for $i \in \{1,2\}$. If $\sqrt{\Delta_{i}(n)} \in \mathbb{Q}(\sqrt{\Delta_{j}(n)})$ for $i,j \in \{1,2\}$ and $i\neq j$ then 
\begin{align*}
    x_{n}+y_{n}\sqrt{\Delta_{j}(n)}&=\sqrt{\Delta_{i}(n)} \ \text{for} \ x_{n},y_{n} \in \mathbb{Q},\\
    x_{n}^2+y_{n}^2\Delta_{j}(n)+2x_{n}y_{n}\sqrt{\Delta_{j}(n)}&=\Delta_{i}(n).
\end{align*}By comparing coefficients $2x_{n}y_{n}=0$, either $x_{n}=0$ or $y_{n}=0$. If $x_{n}=0$, $\Delta_{3}(n)$ is a square in $\mathbb{Q}$ and if $y_{n}=0$ then, $x_{n}^2=\Delta_{i}$ both of which are contradictions to the fact $n\notin A_{1}$. Hence,$$\left[\mathbb{Q}\left(\sqrt{\Delta_{1}(n)},\sqrt{\Delta_{2}}(n)\right):\mathbb{Q}\right] =4$$ then by applying the tower law we can conclude that$\left[\mathbb{Q}\left(r_{1},r_{2},\sqrt{\Delta_{1}(n)},\sqrt{\Delta_{2}(n)}\right):\mathbb{Q}\right]=36$.
\end{proof}
\begin{lemma}\label{Section5MainLemma}
Let $$A_{2}:=\{n\in \mathbb{Z}:\Delta_{1}(n),\Delta_{2}(n)\ \text{or}\ \Delta_{3}(n)\ \text{is a square or}\ f_{2}\ \text{is reducible}\}$$ then for all $n\notin A_{2}$ then the compositum of the splitting field of $f_{1}$ and $f_{2}$, denoted by $K$, is a $S_{3}\times S_{3}$ or $C_{2}\times S_{3}$ extension of $\mathbb{Q}$.
\end{lemma}
\begin{proof}
Fix $n\notin A_{2}$. If $f_{1}$ is irreducible then by Lemma \ref{BothPolyFullExt} $K$ is a $S_{3}\times S_{3}$ extension of $\mathbb{Q}$. If $f_{1}$ is reducible, by a similar argument to Lemma \ref{BothPolyFullExt} $K$ is a $C_{2}\times S_{3}$ extension of $\mathbb{Q}$.
\end{proof}
\begin{proposition}\label{TrivialBrauerGroup}
Fix $n\in \mathbb{Z}$ such that $X_{n}$ is smooth and $n\notin A_{2}$, then $\Br X_{n} = \Br \mathbb{Q}$.
\end{proposition}
\begin{proof}
Note that if $n\not\in A_{2}$ then $4a^3+27(3b-n)^2\neq 0$, as otherwise $f_{2}$ would be reducible. Then by Proposition \ref{prop: changing g1 to f2} the compositum  of the splitting fields of $f_{1}$ and $g_{1}$ is isomorphic to the splitting field of $f_{1}$ and $f_{2}$. Then the statement follows from Proposition \ref{TrivialBrauerGroup1}.
\end{proof}
\begin{notation}
For the rest of the paper denote by $A_{3}:=A_{2}\cup\{n\in \mathbb{Z}:X_{n}\ \text{is not smooth}\}$.
\end{notation}
\section{Finiteness results for thin sets}
In this section we introduce Hilbert's irreducibility theorem in modern terminology, using thin sets; a notion introduced by Serre \cite[Ch. 3 \S 1]{S16}. We show if  we fix $a,b\in \mathbb{Z}$ such that $a$ is non-zero then there exists a thin set of $n\in \mathbb{Z}$ such that $X_{n}$ is a smooth cubic surface over $\mathbb{Q}$ with $\Br X_{n}=\Br \mathbb{Q}$. In Section \ref{RedCount} we show that this thin set is finite. Throughout Section \ref{ThinSets}, $k$ will be a field of characteristic 0 and an algebraic variety over $k$ will be integral and quasi-projective. Denote by $V$ an algebraic variety over $k$.
\subsection{Thin sets}\label{ThinSets}
\begin{definition}
A subset $A\subset V(k)$ is of \emph{type I} if there is a proper closed subscheme $W\subset V$, with $A\subset W(k)$. $A$ is said to be of \emph{type II} if there is an algebraic variety $V'$ with $\dim V=\dim V',$ and a generically surjective morphism $\pi:V'\rightarrow V$ of degree greater than or equal to 2, with $A\subset \pi(V'(k))$.
\end{definition}
\begin{definition}
A subset $A \subset V(k)$ is called \emph{thin} if it is contained in a finite union of sets of type I or II.
\end{definition}
Let $V$ be as above, $k(V)$ be its function field and $$f(x)=x^{n}+a_{n-1}x^{n-1}+\dotsc a_{0}\in k(V)[x]$$ be an irreducible polynomial over $k(V)$. Let $G\subset S_{n}$ be the Galois group of $f$, viewed as a group of permutations on the roots of $f$. If $t\in V(k)$ and $t$ is not a pole of any of the $a_{i}$ then $a_{i}(t)\in k$, and one can define the specialization of $f$ at $t$ as $$
f_{t}(x)=x^{n}+a_{n-1}(t)x^{n-1}+\dotsc a_{0}(t)\in k[x].
$$
{}\begin{theorem}[Hilbert's irreducibility theorem, {\cite[Prop 3.3.5]{S16}}] \label{HilIrr}
There exists a thin set $A \subset V(k)$ such that, if $t\not\in A,$ then \begin{enumerate}
    \item $t$ is not a pole of any of the $a_{i}$,
    \item $f_{t}(x)$ is irreducible over $k$,
    \item the Galois group of $f_{t}$ is $G$.
\end{enumerate}
\end{theorem}

\begin{proposition}\label{ThinSetProp}
Fix $a,b\in \mathbb{Z}$ with $a$ non-zero. Then there exists a thin set $A\subset \mathbb{A}^{1}_{\mathbb{Q}}(\mathbb{Q})$, such that for $n\not\in A$, the cubic surface $X_{n}$ over $\mathbb{Q}$ is smooth and has the property $\Br X_{n} = \Br \mathbb{Q}$.
\end{proposition}
\begin{proof}
By Proposition \ref{SmoothFinite} the set $B:=\{n\in \mathbb{Q}:X_{n}\ \text{is not smooth}\}$ is finite. Using Theorem \ref{HilIrr} there is a thin set $C$ of $n\in \mathbb{Q}$ such that the compositum of the splitting field of $f_{1}$ and $f_{2}$ generate an $S_{3}\times S_{3}$ extension of $\mathbb{Q}$. Then the proposition follows from using Proposition \ref{TrivialBrauerGroup} and letting $A:=(B\cup C)\cap \{n:n\in \mathbb{Z}\}$.
\end{proof}
\begin{remark}
The assumption that $a$ is non-zero in Proposition \ref{ThinSetProp} is a necessary condition as if $a=0$ then by \cite[Prop 2.1]{CTW12} $\Br X_{n}/\Br \mathbb{Q}\cong \mathbb{Z}/3\mathbb{Z}$ for all $n\in \mathbb{Q}$ such that $f_{1}$ is irreducible.
\end{remark}
\subsection{Finiteness results}\label{FinitenessResults}
\label{RedCount}
 Section \ref{RedCount} is dedicated to showing that if $a,b\in \mathbb{Z}$ are fixed and $a$ is non-zero. For all but finitely many choices of $n\in \mathbb{Z}$, $\Br X_{n}=\Br \mathbb{Q}$. By Proposition \ref{TrivialBrauerGroup} it will be sufficient to show that for all but finitely many $n\in \mathbb{Z}$ the polynomials $f_{1}$ and $f_{2}$ generate either a $S_{3}\times S_{3}$-extension or a $C_{2}\times S_{3}$-extension such that $f_{1}$ is reducible.

\begin{lemma}\label{CubicDiscLemma1} $\Delta_{1}(n)$ is a square only for finitely many $n\in \mathbb{Z}$.
\end{lemma}
\begin{proof}
Suppose $\Delta_{1}(n)$ is a square for infinitely many choices of $n\in \mathbb{Z}$ i.e., there exists an infinite family $(n,z_{n})\in \mathbb{Z}^2$ such that 
\begin{equation}\label{Label1DiscProof}
\Delta_{1}(n)=-{\left(4 \, a^{3} + 27 \, {\left(3 \, b - n\right)}^{2}\right)} = z_{n}^2.
\end{equation} It is clear there are only finitely many pairs $(n,0)$ in such a family $(n,z_{n})$, hence we restrict our considerations to $n$ such that $z_{n} \in \mathbb{Z}\backslash \{0\}$. If $a>0$ such a $z_{n}$ cannot exist, if $a<0$ let $a=-a'$ and we can rewrite (\ref{Label1DiscProof}) as 
$4(a^{\prime})^3-27(3b-n)^2=z_{n}^{2}$. Let \begin{align*}
    N:\mathbb{Z}\left[\frac{1+\sqrt{-3}}{2}\right]\rightarrow \mathbb{Z}
\end{align*}be the norm map associated to the field extension $\mathbb{Q}(\sqrt{-3})$. Then the existence of a family $(n,z_{n})$ implies to the existence of infinitely many $(x_{n},y_{n})\in \mathbb{Z}^2$ such that \begin{align}\label{LabelDiscProof3}
    N(x_{n}+\alpha y_{n})=4(a')^3
\end{align} where $\alpha = \frac{1+\sqrt{-3}}{2}$. As $\mathbb{Z}\left[\frac{1+\sqrt{-3}}{2}\right]$ is a unique factorization domain with only finitely many units, there are only finitely many $x_{n},y_{n} \in \mathbb{Z}$ that satisfy (\ref{LabelDiscProof3}), a contradiction. 
\end{proof}
\begin{lemma}\label{CubicDiscLemma2} $\Delta_{2}(n)$ is a square only for finitely many $n\in \mathbb{Z}$.
\end{lemma}
\begin{proof}
Suppose $\Delta_{2}(n)$ is a square for infinitely many choices of $n\in \mathbb{Z}$ i.e., there exists an infinite family $(n,z_{n})\in \mathbb{Z}^2$ such that \begin{align}\label{LabelDisc2Proof1}\Delta_{2}(n) = -243 \, {\left(4 \, a^{3} + 27 \, {\left(3 \, b - n\right)}^{2}\right)} {\left(4 \, a^{3} + 3 \, {\left(3 \, b - n\right)}^{2}\right)}=z_{n}^2.\end{align} As in Lemma \ref{CubicDiscLemma1} we can assume $z_{n}\in \mathbb{Z}\backslash \{0\}$. The statement (\ref{LabelDisc2Proof1}) is equivalent to
\begin{align}\label{LabelDisc2Proof2}
    -(4a^3+27(3b-n)^2)(4a^3+3(3b-n)^2)=3t_{n}^2
\end{align}for $t_{n}\in \mathbb{Z}\backslash \{0\}.$ By expanding brackets, we know that $-(16 \, a^{6} + 120 \, a^{3} (3b-n)^{2} + 81 \, (3b-n)^{4})=3t_{n}^2$. Hence, $3$ divides $a$ and clearly, we require $a<0$ as $t_{n}^2>0$. Let $-a=3a^{\prime}$ then by (\ref{LabelDisc2Proof2}) and dividing through by 3, we obtain $(36(a^{\prime})^3-9(3b-n)^2)((3b-n)^2-36(a^{\prime})^3)=t_{n}^2$. Let $x = (36(a^{\prime})^3-9(3b-n)^2)$ and $y = ((3b-n)^2-36(a^{\prime})^3)$, then for large enough $n$ we have that $y>0$ and $x<0$, hence such an infinite family $(n,z_{n})\in \mathbb{Z}^2$ cannot exist.
\end{proof}
\begin{lemma}\label{CubicDiscLemma3} $\Delta_{3}(n)$ is a square for finitely many $n\in \mathbb{Z}$.
\end{lemma}
\begin{proof} Suppose $\Delta_{3}(n)$ is a square in $\mathbb{Q}$ for infinitely many $n\in \mathbb{Z}$  i.e., there exists an infinite family $(n,z_{n})\in \mathbb{Z}\times \mathbb{Q}$ such that
\begin{equation}\label{Label3DiscProof1}
    \Delta_{3}(n)=\begin{cases}\frac{-243 \, {\left(4 \, a^{3} + 27 \, {\left(3 \, b - n\right)}^{2}\right)} {\left(4 \, a^{3} + 3 \, {\left(3 \, b - n\right)}^{2}\right)}}{-\, {\left(4 \, a^{3} + 27 \, {\left(3 \, b - n\right)}^{2}\right)}}=z_{n}^{2}\ &\text{if} \ 4 \, a^{3} + 27 \, {\left(3 \, b - n\right)}^{2}\neq 0,\\
    0 \ &\text{if} \ 4 \, a^{3} + 27 \, {\left(3 \, b - n\right)}^{2}=0.\end{cases}
\end{equation} Clearly, $\Delta_{3}(n)=0$ for only finitely many choices of $n$, so we can focus on the cases where $z_{n} \in \mathbb{Q}^{\times}$. Then (\ref{Label3DiscProof1}) is equivalent to
$\Delta_{3}(n)={{4 \, a^{3} + 3 \, {\left(3 \, b - n\right)}^{2}}}=3t_{n}^2$ where $t_{n}\in \mathbb{Z}\backslash\{0\}$. Clearly $3$ divides $a$, hence ${{36 \, (a')^{3} +  \, {\left(3 \, b - n\right)}^{2}}}=t_{n}^2$ where $a^{\prime}:=a/3$. Let $N:=36 \, (a')^{3}$ and $x_{n}:=3 \, b - n$ then 
$t_{n}^2-x_{n}^2=N$ and as $N\in \mathbb{Z}$ it has only finitely many factors, so there are only finitely many choices for $x_{n}$ and $t_{n}.$
\end{proof}
\begin{lemma}\label{SmoothCurveFinThin}
Fix $a,b\in \mathbb{Z}$ with $a$ non-zero. Denote by $C$ the affine curve $$
C: x^3 - 12ax^2 + 36a^2x + 27(3b-t)^2 + 4a^3=0 \subset \mathbb{A}^{2}_{\mathbb{Q}}
$$ in variables $t$ and $x$. Then the compactification $\tilde{C}$ of $C$ is smooth.
\end{lemma}
\begin{proof}
Let $F(t,x,y)=x^3-12ax^2y+36a^2xy^2+27y(3by-t)^2+4a^3y^3$, then the partial derivatives of $F$ are $
\frac{\partial F}{\partial t}=-54 \, {\left(3 \, b y - t\right)} y, \ \frac{\partial F}{\partial x}=36 \, a^{2} y^{2} - 24 \, a x y + 3 \, x^{2},\ \frac{\partial F}{\partial y}=12 \, a^{3} y^{2} + 72 \, a^{2} x y - 12 \, a x^{2} + 162 \, {\left(3 \, b y - t\right)} b y + 27 \, {\left(3 \, b y - t\right)}^{2}$. Assume that $\tilde{C}$ has a non-smooth point $p=[t,x,y]$. Then $F(p)=0$ and all the partial derivatives of $F$ will vanish at $p$. As $\frac{\partial F}{\partial t}(p)=0$, then either $y=0$ or $t=3by$. If $y=0$ clearly $x=0$, as $\frac{\partial F}{\partial y}(p)=x=y=0$ this implies $t=0$ which is a contradiction. Suppose $t=3by$, as $\ \frac{\partial F}{\partial x}(p)=0$ this implies $x=6ay$ or $x=2ay$. However, as we require $F(p)=0$, if $t=3by$ and $x=6ay$ or $x=2ay$ this implies $a=0$ which is a contradiction. Hence, $\tilde{C}$ is smooth.
\end{proof}
\begin{proposition}\label{ApplySigels}
Let $f(t,x)=x^3-12ax^2+36a^2x+27(3b-t)^2+4a^3$ such that $a,b\in \mathbb{Z}$ and $a$ is non-zero. Then for all but finitely many $n \in \mathbb{Z}$, the polynomial $f(n,x)\in \mathbb{Z}[x]$ is irreducible.
\end{proposition}
\begin{proof}
By Lemma \ref{SmoothCurveFinThin} and the genus-degree formula, the curve $f(t,x)\subset\mathbb{A}^{2}_{\mathbb{Q}}$ has genus 1. Every specialisation $t=n\in \mathbb{Z}$ gives rise to a monic cubic polynomial $f(n,x)\in \mathbb{Q}[x]$. By Gauss's Lemma \cite[Thm 2.3]{S02}, we see that $f(n,x)$ is irreducible over $\mathbb{Q}[x]$ if and only if $f(n,x)$ is irreducible over $\mathbb{Z}[x]$. Hence, we can apply Siegel’s Theorem on curves \cite[Thm 3.2]{SC14} to show there are only finitely many specialisations $t=n\in \mathbb{Z}$ such that $f(n,x)$ is reducible.
\end{proof}
\begin{corollary}\label{AllbutfinTrivialBrauer} Fix $a,b\in \mathbb{Z}$ such that $a$ is non-zero then for all but finitely many choices of $n\in \mathbb{Z}$, $\Br X_{n}= \Br \mathbb{Q}$.
\end{corollary}
\begin{proof}
By Proposition \ref{TrivialBrauerGroup} it is sufficient to show the set $A_{3}$ is finite, this follows from Lemmas \ref{CubicDiscLemma1},\ref{CubicDiscLemma2}, \ref{CubicDiscLemma3} and Propositions \ref{SmoothFinite}, \ref{ApplySigels}.
\end{proof}
\section{Geometry of the affine surfaces}\label{GeometryOfAffine}
When studying the geometry of the affine cubic surfaces $\mathcal{U}_{n}$, one notices that the geometry is similar in certain aspects to that of the sum of three cubes. An example of this is, the presence of transcendental Brauer group elements in certain cases. We adapt the argument given by Wittenberg and Colliot-Thélène's in \cite[Prop 3.1]{CTW12} to prove that for certain cases there are no transcendental elements in the $\Br U_{n}$. Furthermore, the Brauer group of the compactification of $U_{n}$ coincides with $\Br U_{n}$. 
\subsection{Brauer group of the affine surfaces}
\begin{lemma}[{\cite[Lemma 3.2]{CTW12}}]\label{FermatCubic}
Let $D \subset \mathbb{P}^{2}_{\mathbb{Q}}$ be a plane curve with equation $x_{1}^3+x_{2}^3+x_{3}^3=0$ and denote by $P_{1},P_{2},P_{3} \in D(\mathbb{Q})$ the points with coordinates $P_{1}=[1:-1:0]$, $P_{2}=[-1:0:1]$, $P_{3}=[0:-1:1]$. Then
\begin{enumerate}
    \item The image of $\HH_{\et}^{1}(D,\mathbb{Q}/\mathbb{Z}) \rightarrow \HH_{\et}^{1}(\mathbb{Q},\mathbb{Q}/\mathbb{Z})$ evaluated at the point $P_{1}$ is isomorphic to $\mathbb{Z}/3\mathbb{Z}$.
    \item An element of $\HH_{\et}^{1}(D,\mathbb{Q}/\mathbb{Z})$ which is zero at $P_{1}$ and $P_{2}$ is zero.
\end{enumerate}
\end{lemma}
\begin{proposition}\label{AffineBrauerGroup} Fix choice of $n \in \mathbb{Z}$ such that $f_{1}$ has a non-square discriminant and the cubic surface $X_{n}$ is smooth. Then the natural map $\Br X_{n} \rightarrow \Br U_{n}$ is an isomorphism.
\end{proposition}
\begin{proof}Denote by $D \subset X$ the hyperplane section defined by $x_{0}=0$. Applying Grothedieck's purity theorem \cite[Thm 3.7.1]{CTS21}, gives an exact sequence
$$
    0\rightarrow \Br X_{n} \rightarrow \Br U_{n} \xrightarrow[]{\partial_{D}} \HH_{\et}^{1}(D,\mathbb{Q}/\mathbb{Z}).
$$ We want to show that the image of the residue map $\partial_{D}$ is 0. Let $A\in \Br U_{n}$ but $A\notin \Br X_{n}$ and let $m = \partial_{D}(A) \in \HH_{\et}^{1}(D,\mathbb{Q}/\mathbb{Z})$. Denote by $P_{1},P_{2},P_{3}$ the points in $X_{n}(\mathbb{Q})$ with coordinates in $[x_{0}:x_{1}:x_{2}:x_{3}]$ given by $P_{1}=[0:1-1:0], P_{2}=[0:-1:0:1], P_{3}=[0:0:-1:1].$ 
Let $k=\mathbb{Q}[x]/(f_{1}(x))$ where $r$ is the image of $x$ and $L_{i}:=L_{i,1}\subset X_{n}\times k$ for $i=1,2,3$. Each line $L_{i}$ on $X_{n} \times k$ intersects the divisor $D \times k$ at a unique point $P'_{i}:=P_{i} \times k$. This gives the following commutative diagram for each $i$
\[ \begin{tikzcd}
0 \arrow{r} &\Br X_{n} \arrow{r} \arrow[swap]{d} & \Br U_{n} \arrow{d} \arrow{r}{\partial_{D}} & \HH_{\et}^{1}(D,\mathbb{Q}/\mathbb{Z}) \arrow{d}{\alpha_{i}} \\%
0 \arrow{r} &\Br L_{i} \arrow{r}& \Br (L_{i} \backslash P'_{i}) \arrow{r} & \HH_{\et}^{1}(P'_{i},\mathbb{Q}/\mathbb{Z})
\end{tikzcd}
\] Let $A_{i} \in \Br (L_{i}\backslash P'_{i})$ be the restriction of $A$ to $L_{i}\backslash P'_{i}$. As $L_{i}\backslash P'_{i} \cong \mathbb{A}_{k}^{1}$ we have that $\Br(L_{i}\backslash P'_{i}) = \Br(k)$ hence the class of $A_{i}$ is constant, and it's residue at the point $P'_{i}$ is zero. One can now deduce that $\im(\partial_{D}) \subset \ker(\alpha_{i})$ for all $i$. As $P_{i}\in X_{n}(\mathbb{Q})$ we have that $m(P_{i})\in \HH^{1}_{\et}(\mathbb{Q},\mathbb{Q}/\mathbb{Z})$ then the image if $m(P_{i})$ under the restriction map
$$
    \Res_{k/\mathbb{Q}}:\HH_{\et}^{1}(\mathbb{Q},\mathbb{Q}/\mathbb{Z})\rightarrow \HH_{\et}^{1}(k,\mathbb{Q}/\mathbb{Z})
$$ is $m(P'_{i})$. The map $\Res_{k/\mathbb{Q}}$ is injective as $k$ is either trivial or a non-Galois cubic extension of $\mathbb{Q}$, as we require $m(P'_{i})=0$ for $i=1,2,3$ the injectivity of the restriction map implies $m(P_{i})=0$. Using Lemma \ref{FermatCubic} we can deduce that $m=0$.
\end{proof}
\section{Main theorems}\label{MainThms}
\begin{theorem}\label{GeneralCaseThm}
Let $f(x)=u^3+au+b \in \mathbb{Z}[u]$ where $a$ is non-zero and let $n\in \mathbb{Z}$ such that $n\notin A_{3}$. Consider the affine surface $\mathcal{U}_{n}$ over $\mathbb{Z}$ defined by
$$
    \mathcal{U}_{n}:f(u_{1})+f(u_{2})+f(u_{3})=n. 
$$ Denote $U_{n}=\mathcal{U}_{n} \times\mathbb{Q}$, then $\Br U_{n}=\Br {\mathbb{Q}}$ and there is no-integral Brauer-Manin obstruction to the Hasse principle on $\mathcal{U}_{n}$.
\end{theorem}
\begin{proof}
In Proposition \ref{AllbutfinTrivialBrauer} we showed that if $n\notin A_{3}$ then $\Br X_{n}=\Br \mathbb{Q}$. Furthermore, it follows from Proposition \ref{AffineBrauerGroup} for such choices of $n$ we have $\Br U_{n}= \Br X_{n}=\Br \mathbb{Q}$. If $\mathcal{U}_{n}$ is not everywhere locally soluble for some $n$ we automatically have
$$
    \mathcal{U}_{n}(\mathbb{Z})=\emptyset \  \text{and} \  \mathcal{U}_{n}(\mathbb{A}_{\mathbb{Z}}) = \emptyset.
$$
In this case we say there is no integral Brauer-Manin obstruction to the Hasse principle. We now assume $\mathcal{U}_{n}$ is everywhere locally soluble. As $\Br U_{n}=\Br \mathbb{Q}$ the statement follows.
\end{proof}
\begin{remark}
If $n\in A_{3}\backslash A_{1}$ then $f_{1}$ is reducible. This implies $\mathcal{U}_{n}(\mathbb{Z})\neq 0$, hence in this case the integral Hasse principle holds.
\end{remark}
\begin{theorem}\label{WCLTCASE}
Let $f(u)=u^3+a_{2}u^2+a_{1}u+a_{0} \in \mathbb{Z}[u]$ and consider the affine cubic surface over $\mathbb{Z}$ $$
    \mathcal{U}_{n}:f(u_{1})+f(u_{2})+f(u_{3})=n
$$for some $n\in \mathbb{Z}$. If $3a_{1}-a_{2}^2=0$ then there is no integral Brauer-Manin obstruction to the Hasse principle.
\end{theorem}
\begin{proof}
Suppose $a_{2}=0$, as $3a_{1}-a_{2}^2=0$ then $a_{1}=0$, hence we can write $\mathcal{U}_{n}$ as 
\begin{equation}\label{WCLTCASE1}
\mathcal{U}_{n}:u_{1}^3+u_{2}^3+u_{3}^{3}=n-3a_{0}.
\end{equation} If $a_{2}\neq 0$, $3$ divides $a_{2}$ and the change of variables $u_{i}\rightarrow u_{i}-\frac{a_{2}}{3}$ for $i=1,2,3$ defines an isomorphism $\mathcal{U}_{n}\rightarrow \mathcal{U}^{\prime}_{n}$ over $\mathbb{Z}$, where
\begin{equation}\label{WCLTCASE2}
    \mathcal{U}^{\prime}_{n}:u_{1}^{3}+u_{2}^{3}+u_{3}^{3}=n-2a_{2}-a_{2}a_{1}+3a_{0}.
\end{equation} Then by \cite[Thm 4.1a]{CTW12} (\ref{WCLTCASE1}) and (\ref{WCLTCASE2}) never have an integral Brauer-Manin obstruction for any choice of $n\in \mathbb{Z}$.
\end{proof}
\begin{proof}[\textbf{Proof of Theorem \ref{BIGTHM}}]
Let $f(u)=u^3+a_{2}u^2+a_{1}u+a_{0} \in \mathbb{Z}[u]$ and consider the affine cubic surface over $\mathbb{Z}$ $$
    \mathcal{}{U}_{n}:f(u_{1})+f(u_{2})+f(u_{3})=n
$$for some $n\in \mathbb{Z}$. By Theorem \ref{WCLTCASE} we may assume $3a_{1}-a_{2}^2\neq 0$. Denote $U_{n}:=\mathcal{U}_{n}\times_{\mathbb{Z}} \mathbb{Q}$, apply the change of variables $u_{i}\rightarrow \frac{u_{i}-a_{2}}{3}$  and clearing denominators, we can write $U_{n}$ as
\begin{align}\label{ChangeOfVar3}
U_{n}:u_{1}^{3}+u_{2}^{3}+u_{3}^{3}+a(u_{1}+u_{2}+u_{3})=27n-3(2a_{2}^3-9a_{1}a_{2}+27a_{0})
\end{align} where $a:=9a_{1}a_{3}-3a_{2}^2\in \mathbb{Z}\backslash\{0\}$. Then the statement follows from Theorem \ref{GeneralCaseThm} and Corollary \ref{AllbutfinTrivialBrauer}.
\end{proof}
\section{Sum of three tetrahedral numbers}\label{Examples}
\subsection{Sum of three tetrahedral numbers}\label{SumOfThreeTetSection}
Zhi-Wei Sun asks in \cite{325659} if any integer $n$ can be represented as the sum of three tetrahedral numbers. We show that for any choice of $n$, such an equation is everywhere locally soluble and there is no integral Brauer-Manin obstruction. Throughout Section \ref{SumOfThreeTetSection} $k$ is a field of characteristic not equal to 2 or 3 such that $3$ is a square in $k$. Let $F(x_{i},x_{0})=x_{i}(x_{i}+x_{0})(x_{i}+2x_{0})\in k[x_{i},x_{0}]$ and $p_{0},p_{1},p_{2},p_{3}$ be the geometric points on $\mathbb{P}^{3}_{k}$
\begin{align*}
    p_{0}&: x_{1} = x_{2} = x_{3} = \left(-1 +\frac{\sqrt{3}}{3}\right)x_{0}, &&p_{2}: x_{1}=x_{2}=\left(-1 +\frac{\sqrt{3}}{3}\right)x_{0}, x_{3} =\left(-1 -\frac{\sqrt{3}}{3}\right)x_{0},\\
    p_{1}&: x_{1} = x_{2} = x_{3} = \left(-1 -\frac{\sqrt{3}}{3}\right)x_{0},
    \ &&p_{3}: x_{1}=x_{2}=\left(-1 -\frac{\sqrt{3}}{3}\right)x_{0}, x_{3} =\left(-1 +\frac{\sqrt{3}}{3}\right)x_{0}.
\end{align*} To begin with, we consider the family of cubic surfaces $$
 X_{n}: F(x_{1},x_{0})+F(x_{2},x_{0})+F(x_{3},x_{0})-6nx_{0}^3 \subset \mathbb{P}^3_{\mathbb{Q}}
$$
 where $n\in \mathbb{Q}$. 
\begin{lemma}\label{SmoothSumOfTetLemma} Denote by $X$ the surface $$
X:F(x_{0},x_{1})+F(x_{0},x_{2})+F(x_{0},x_{3})=6nx_{0}^3
$$ defined over $k$. If $X$ has a singular point then either $n =\frac{\sqrt{3}}{9},\frac{-\sqrt{3}}{9}, \frac{\sqrt{3}}{27}$ or $\frac{-\sqrt{3}}{27}$.
\end{lemma}
\begin{proof}
Let $G(x_{0},x_{1},x_{2},x_{3}):=F(x_{0},x_{1})+F(x_{0},x_{2})+F(x_{0},x_{3})-6nx_{0}^3$. Assume there is a non-smooth point $p = [x_{0}:x_{1}:x_{2}:x_{3}]$ on $X_{n}$, then the partial derivatives of $G$ will vanish at $p$ i.e.\ $$\frac{\partial G}{\partial x_{i}}(p) = 2x_{0}^2 + 6x_{0}x_{i} + 3x_{i}^2=0, \frac{\partial G}{\partial x_{0}}(p) = -18nx_{0}^2 + 4x_{0}x_{1} + 4x_{0}x_{2} + 4x_{0}x_{3} +3x_{1}^2+ 3x_{2}^2+ 3x_{3}^2=0
$$ for $i=1,2,3$. We can see if $x_{0}=0$ for such a point $p$ this would imply $x_{i}=0$ for $i=1,2,3$ which cannot happen, hence we can assume $x_{0}\neq 0$. Using the symmetry of the defining equation for $X$ and the fact $\frac{\partial G}{\partial x_{i}}(p)=0$ we can assume that $p$ is either $p_{0},p_{1},p_{2}$ or $p_{3}$. As $\frac{\partial G}{\partial x_{0}}(p)=0$ this implies $n =\frac{\sqrt{3}}{9},\frac{-\sqrt{3}}{9}, \frac{\sqrt{3}}{27}$ or $\frac{-\sqrt{3}}{27}$.
\end{proof}
\begin{proposition}\label{SmootnessOfTet}
$X_{n}$ is smooth for all $n \in \mathbb{Q}$. 
\end{proposition}
\begin{proof}
This follows from Lemma \ref{SmoothSumOfTetLemma}.
\end{proof}
\subsection{Finding $\mathbb{Z}_p$ points on $\mathcal{U}_{n}$}\label{PadicTet} Throughout Section \ref{PadicTet}, $p$ will be a prime and $\mathbb{F}_{q}$ a finite field of size $q=p^{k}$ for some integer $k$. Denote by $\mathcal{X}_{n}$ the integral model of $X_{n}$ defined by the same equation over $\mathbb{Z}_{p}$ and denote by $\mathcal{X}_{n,p}$ the special fibre of $\mathcal{X}_{n}$. If the special fibre is smooth we say $\mathcal{X}_{n}$ has \emph{good reduction} at $p$, otherwise we say $\mathcal{X}_{n}$ has \emph{bad reduction} at $p$. 
\begin{theorem}[Weil, {\cite[Thm 27.1]{Man86}}]\label{WeilPoints}
Let $X$ be a smooth surface over $\mathbb{F}_{q}$, if $\bar{X}$ is rational, then $$
\#X(\mathbb{F}_{q})=q^2+\Tr(\phi^{*})q+1
$$ where $\Tr\phi^{*}$ denotes the trace of the Frobenius automorphism $\phi$ in the representation of $\Gal(\bar{\mathbb{F}}_{q}/\mathbb{F}_{q})$ on $\Pic(\bar{V})$.
\end{theorem}
\begin{remark}
For a smooth cubic surface $X$ over a field $k$ we have that $\Pic \bar{X}\cong \mathbb{Z}^{7}$. If $k=\mathbb{F}_{q}$ it is then clear that $|\Tr(\phi^{*})|\leq 7$. Thus, for a smooth cubic surface $X$ over $\mathbb{F}_{q}$ we have a lower bound $\#X(\mathbb{F}_{q})\geq 1-7q+q^2$.
\end{remark}
\begin{theorem}[Hasse, {\cite[Thm V.1.1]{SIL09}}]\label{HasseBound}
Let $E$ be an elliptic curve over $\mathbb{F}_{q}$ then $$
\#E(\mathbb{F}_{q})\leq 2\sqrt{q}+q+1.
$$
\end{theorem}
\begin{lemma}\label{pdividesn} For all primes $p$ which divide $6n$, $\mathcal{X}_{n,p}$ has a smooth $\mathbb{F}_{p}$-point.
\end{lemma}
\begin{proof}
Consider the defining equation for $\mathcal{X}_{n,p}$
$$
    G(x_{0},x_{1},x_{2},x_{3}) = F(x_{0},x_{1})+F(x_{0},x_{2})+F(x_{0},x_{3})\equiv0 \ (\text{mod} \ p).
$$
The point $[x_{0}:x_{1}:x_{2}:x_{3}]=[1:-1:0:0]$ is a smooth $\mathbb{F}_{p}$-point.
\end{proof}
\begin{lemma}\label{CaseByCasensingular}
Let $p>3$ and $p\equiv 2$ mod $3$ then if
\begin{enumerate}
    \item $n \equiv \frac{-\sqrt{3}}{9}$ mod $p$ the point $[1:-1 +\frac{\sqrt{3}}{3}:-1 +\frac{\sqrt{3}}{3}:-1 -\frac{2\sqrt{3}}{3}]\in \mathcal{X}_{n,p}(\mathbb{F}_{p})$ is smooth, 
    \item $n \equiv \frac{\sqrt{3}}{9}$ mod $p$ the point $[1:-1 -\frac{\sqrt{3}}{3}:-1 -\frac{\sqrt{3}}{3}:-1 +\frac{2\sqrt{3}}{3}]\in \mathcal{X}_{n,p}(\mathbb{F}_{p})$ is smooth, 
    \item $n \equiv \frac{\sqrt{3}}{27}$ mod $p$ the point $[1:-1 -\frac{\sqrt{3}}{3}:-1 +\frac{\sqrt{3}}{3}:-1 +\frac{2\sqrt{3}}{3}]\in \mathcal{X}_{n,p}(\mathbb{F}_{p})$ is smooth, 
    \item $n \equiv \frac{-\sqrt{3}}{27}$ mod $p$ the point $[1:-1-\frac{\sqrt{3}}{{3}}:-1+\frac{\sqrt{3}}{3}:-1-\frac{2\sqrt{3}}{3}] \in \mathcal{X}_{n,p}(\mathbb{F}_{p})$ is smooth.
\end{enumerate}
\end{lemma}
\begin{proof}
Denote by $F$ the defining equation for $\mathcal{X}_{n,p}$, then $\frac{\partial F}{\partial x_{1}}\equiv 3$ mod $p$ for each point.
\end{proof}
\begin{lemma} \label{BadPrimesWhichDontDividen} 
Let $p>3$ and $p$ be a prime of bad reduction. Then there exists a smooth $\mathbb{F}_p$-point on $\mathcal{X}_{n,p}$.
\end{lemma}
\begin{proof}
Using Lemma \ref{SmoothSumOfTetLemma} we see that for $p$ to be a bad prime $3$ needs to be square mod $p$ this is equivalent to $p\equiv 2$ mod $3$ and $n$ mod $p$ has four possible forms which are dependent on $\sqrt{3}$ mod $p$, namely $n\equiv \frac{-\sqrt{3}}{9}, \frac{\sqrt{3}}{9}, \frac{\sqrt{3}}{27}$ or $ \frac{-\sqrt{3}}{27}$ mod $p$. By Lemma \ref{CaseByCasensingular} in each case there exists a smooth $\mathbb{F}_{p}$ point on $\mathcal{X}_{n,p}$ by Lemma \ref{CaseByCasensingular}.
\end{proof}
\begin{lemma}\label{Goodprimes}
Let $p\geq 11$ and a prime of good reduction. Then $\mathcal{X}_{n,p}$ has a smooth $\mathbb{F}_{p}$-point which does not lie on the prime divisor $x_{0}=0$.
\end{lemma}
\begin{proof} As $p$ is prime of good reduction we can apply Theorem \ref{WeilPoints}
$$
    \#\mathcal{X}_{n,p}(\mathbb{F}_p)=1+\Tr \phi^{*}p+p^2\geq1-7p+p^2.
$$  Denote by $\mathcal{X}_{n,p}'(\mathbb{F}_p)$ the points away from the hypersection defined by $x_{0}=0$, using Theorem \ref{HasseBound}, along the divisor $x_{0}=0$ to get we get that
$$
    \#\mathcal{X}_{n,p}'(\mathbb{F}_p) \geq 1 -7p+p^2-(1+p+2\sqrt{p}).
$$ It is easy to check that for $p\geq 11$, $\mathcal{X}_{n,p}'(\mathbb{F}_p) \geq 1$.
\end{proof} 
\begin{proposition}\label{SmoothFpPointProp}
There exists a smooth $\mathbb{F}_{p}$-point on the special fibre of $\mathcal{X}\rightarrow \Spec \mathbb{Z}_{p}$ which lies away from the divisor $x_{0}=0$.
\end{proposition}
\begin{proof}
For primes $p\geq 11$, $p=2$ or $3$ this follows from Lemmas \ref{pdividesn}, \ref{BadPrimesWhichDontDividen} and \ref{Goodprimes}. The only cases left to consider are $p=5$ or $7$. Let $p=5$. If $n = 1,4$ mod $5$ then there exists a smooth point $[1:x:-x:0]$ where $x^2=n$ mod 5. If $n=2$ mod $5$ then $[1:1:1:0]$ is a smooth point and if $n=3$ mod $5$ then $[1:2:2:0]$ is a smooth point. Let $p=7$. If $n = 1,2,4$ mod $7$ then there exists a smooth point $[1:x:-x:0]$ where $x^2=n$ mod 7. If $n=3$ mod $7$ then $[1:3:0:0]$ is a smooth point, if $n=5$ mod $7$ then $[1:4:4:0]$ is a smooth point and if $n=6$ mod 7 then $[1:3:3:0]$ is a smooth point.
\end{proof}
\begin{theorem}\label{SumOfThreeTetELS}
Let $v$ be a place of $\mathbb{Q}$ then $\mathcal{U}_{n}(\mathbb{Z}_{v})\neq \emptyset$.
\end{theorem}
\begin{proof} For the real place $v$, clearly $\mathcal{U}_{n}(\mathbb{Z}_{v})\neq 0$. Let $p$ be a prime. Then by Proposition \ref{SmoothFpPointProp} we see that $\mathcal{X}_{n,p}$ has a smooth $\mathbb{F}_{p}$-point away from the divisor $x_{0}=0$. Using Hensel's lemma we can lift this point to a $\mathbb{Z}_p$-point $\mathcal{X}_{n}$ which lies in the image of $\mathcal{U}_{n}(\mathbb{Z}_{p})\hookrightarrow \mathcal{X}_{n}(\mathbb{Z}_{p})$.
\end{proof}
\subsection{Brauer group of $U_{n}$}
Using Proposition \ref{SplittingFieldProp} we see that the splitting field $K$ of $X_{n}$ for any $n$ is defined by the cubics $f_{1}(x)=x^3-x-6n$ and $f_{2}(x)=x^3 + 12x^2 - 36x +972n^2 - 4$. We can now determine the Brauer group of $\mathcal{U}_{n}\times_{\mathbb{Z}}\mathbb{Q}$.
\begin{lemma}\label{SumOfThreeTetSplitting}
Fix $n\in \mathbb{Z}$, denote by $K$ the compositum of the splitting fields of $f_{1}$ and $f_{2}$. If $f_{1}$ is irreducible then $[K:\mathbb{Q}]=36$.
\end{lemma}
\begin{proof}
As $f_{2}$ has no solutions modulo 27 and is a degree 3 polynomial we have that, $f_{2}$ is irreducible for any choice of $n\in \mathbb{Z}$. In addition to this, $f_{2}$ has a discriminant equal to 
$-3888 \, {\left(243 \, n^{2} - 1\right)} {\left(27 \, n^{2} - 1\right)}$, which is never a square for any $n \in \mathbb{Z}$. Moreover, $f_{1}(x)$ has discriminant $ 4(1-243n^2)$, which is also never a square for any choice of $n\in \mathbb{Z}$ and the quotient of the discriminants of $f_{1}$ and $f_{2}$ is never a square, as it is not a square modulo 3, for any $n\in \mathbb{Z}$. Hence, $[K:\mathbb{Q}]=36$.
\end{proof}
\begin{proposition}\label{SumOfThreeTetBraur}
Fix $n \in \mathbb{Z}$ such that $f_{1}$ is irreducible then $\Br U_{n}=\Br X_{n}=\Br \mathbb{Q}$.
\end{proposition}
\begin{proof}
Using Lemma \ref{SumOfThreeTetSplitting} and Proposition \ref{TrivialBrauerGroup} then $\Br X_{n}=\Br \mathbb{Q}$. As the discriminant of $f_{1}$ is never a square for any $n \in \mathbb{Z}$ we can apply Proposition \ref{AffineBrauerGroup} to deduce that $\Br X_{n}=\Br U_{n}$.
\end{proof}
\begin{theorem}
There is no integral Brauer-Manin obstruction to the Hasse principle on the affine surface
\begin{align*}
    \mathcal{U}_{n}: u_{1}(u_{1}+1)(u_{1}+2)+u_{2}(u_{2}+1)(u_{2}+2)+u_{3}(u_{3}+1)(u_{3}+2)=6n
\end{align*} for all $n \in \mathbb{Z}$.
\end{theorem}
\begin{proof}
If $f_{1}(u)=u^3-u-6n$ is reducible by Gauss's Lemma \cite[Thm 2.3]{S02} there exists $\alpha\in \mathbb{Z}$ such that $f_{1}(\alpha)=0$. Hence, $(u_{1},u_{2},u_{3})=(\alpha+1,0,0)\in \mathcal{U}_{n}(\mathbb{Z})$ i.e. there is no integral Brauer-Manin obstruction in this case. Suppose $f_{1}$ is irreducible, by Proposition \ref{SumOfThreeTetBraur} $\Br U_{n}=\Br \mathbb{Q}$, and by Theorem \ref{SumOfThreeTetELS} we have $\mathcal{U}_{n}(\mathbb{A}_{\mathbb{Z}}) \neq \emptyset$, hence $\mathcal{U}_{n}(\mathbb{A}_{\mathbb{Z}})^{\Br}=\mathcal{U}_{n}(\mathbb{A}_{\mathbb{Z}}) \neq \emptyset.$
\end{proof}
\begin{remark}
Even in the case where $f_{1}$ is reducible we still have $\Br X_{n}=\Br U_{n}=\Br \mathbb{Q}$ by Propositions \ref{TrivialBrauerGroup},\ref{AffineBrauerGroup}. In regards to the integral Brauer-Manin this fact is useless as you have an obvious integral point, however if one was studying weak approximation on the surfaces $U_{n}$ this does become useful, see Remark \ref{rem: CT conjecture WA}.
\end{remark}
\subsection{Rationality of the sum of three tetrahedral numbers} 
For a more complete understanding of the equation defined by the sum of three tetrahedral we show that these surfaces are non-rational when there is not an obvious integral solution.
\begin{definition}
We say that a surface $X$ is $k$-rational if there exists a birational map $\mathbb{P}^{m} \dashrightarrow X$ defined over $k$, for some positive $m$.
\end{definition}
\begin{definition}
A smooth cubic surface over a field $k$ is said to be $k$-minimal if there is no nonempty $\Gal(\bar{k}/k)$-stable set of pairwise non-intersecting exceptional curves.
\end{definition}
\begin{theorem}[{\cite[Thm 37.1]{Man86}}]\label{thm: minimal cubic surface not rational}
Every minimal smooth cubic surface is non-rational.
\end{theorem}
\begin{lemma}\label{Non-rational} If $n$ is an integer such that $f_{1}=x^3-x-6n$ is irreducible then,
$X_{n}$ is a non-rational surface over $\mathbb{Q}$.
\end{lemma}
\begin{proof}
By Theorem \ref{thm: minimal cubic surface not rational} it is sufficient to prove for such a choice of $n$ the cubic surface $X_{n}$ is minimal. $X_{n}$ has orbit type $[3,3,3,18]$ where the Galois orbits of size $3$ meet at an Eckardt point (see Remark \ref{rem: label lines}) i.e. they are not pairwise skew. The Galois orbit of size 18 cannot be pairwise skew as any set containing more than 6 lines of a smooth cubic surface contains intersecting lines.
\end{proof}
\begin{theorem}
If $n$ is an integer such that $f_{1}=x^3-x-6n$ is irreducible then $U_{n}$ is non-rational over $\mathbb{Q}$.
\end{theorem}
\begin{proof}
$U_{n}$ is birational to $X_{n}$ as it is a Zariski open subset of $X_{n}$. Using Lemma \ref{Non-rational} we can deduce $U_{n}$ is non-rational. 
\end{proof}

\section{Counter-example to weak and strong approximation}\label{SectionSA}
Even though in Theorem \ref{BIGTHM} we showed the surfaces $\mathcal{U}_{n}$ almost never have an integral Brauer-Manin obstruction the surface $U_{n}:=\mathcal{U}_{n}\times_{\mathbb{Z}}\mathbb{Q}$ may still fail strong approximation. We show an example where $U_{n}$ fails weak approximation, hence fails strong approximation. In Section \ref{SectionSA} by a variety over a field $k$ we mean a smooth, separated, geometrically integral scheme of finite type over a field $k$. Throughout let $k$ be a number field and $U$ a variety over $k$. For $S$ any finite set of places of $k$, denote by \[
U(\mathbb{A}_{k}^{S}):= \{(x_{v})_{v}\in \prod\limits_{v\notin S}U(k_{v}): v(x_{v})\geq 0\ \text{for all but finitely many}\ v \}.
\] $U(\mathbb{A}_{k}^{S})$ has a natural restricted product topology. 
\subsection{Weak and strong approximation} We now recall the definition of weak and strong approximation and the relationship between them.
\begin{definition}[{\cite[\S3, p.12]{CTX09}}]
We say $U$ satisfies \emph{strong approximation} off a finite set of place $S\subset \Omega_{k}$ if the diagonal image $U(k)\hookrightarrow U(\mathbb{A}^{S}_{k})$ is dense. Equivalently, if $U_{n}(k)\cap W\neq \emptyset$ for all non-empty open subsets $W \subset U_{n}(\mathbb{A}_{k}^{S})$. If the diagonal image of $U(k)\hookrightarrow U(\mathbb{A}_{k})$ is dense, we simply say $U$ satisfies strong approximation.
\end{definition}
\begin{definition}
We say $U$ satisfies weak approximation off $S$ a finite set of place $S\subset \Omega_{k}$ if the diagonal image \[U(k)\hookrightarrow \prod\limits_{v\in \Omega_{k}\backslash S}U(k_{v})\] is dense. If the diagonal image $U(k)\hookrightarrow \prod_{v\in \Omega_{k}}U(k_{v})$ is dense, we simply say $U$ satisfies weak approximation .
\end{definition}
\begin{remark}
\begin{enumerate}
\item It is clear by definition that strong approximation implies weak approximation.
\item If $X$ is a projective variety then $X(\mathbb{A}_{k})=\prod_{v\in \Omega_{k}} X(k_{v})$, hence weak approximation is equivalent to strong approximation in this case.
\end{enumerate}
\end{remark}
\begin{lemma}[{\cite[Prop 13.2.3]{CTS21}}]\label{lemma: WA birational invariant}
If $U$ and $U'$ are birational varieties over $k$ such that $U$ and $U'$ are everywhere locally soluble then $U$ satisfies weak approximation if and only if $U'$ satisfies weak approximation.
\end{lemma}
\begin{remark}
It should be noted that strong approximation is not a birational invariant. For example if one takes $U=\mathbb{A}^{1}_{\mathbb{Q}},U':=\mathbb{G}_{m,\mathbb{Q}}$ and $S\neq \emptyset$ then $U$ satisfies strong approximation off $S$ whereas $U'$ does not. The failure of strong approximation on $U'$ can be explained by a result of Minčhev \cite{Min89} and the fact that $\pi^{\et}(\bar{U}')=\hat{\mathbb{Z}}$.
\end{remark}

 \begin{remark}\label{rem: CT conjecture WA} Let $f(u)=u^3+a_{2}u^2+a_{1}u+a_{0}$ such that $3a_{1}-a_{2}^2\neq 0$ and denote by $U_{n}$ be the surface \[
U_{n}:f(u_{1})+f(u_{2})+f(u_{n})=n\subset \mathbb{A}^{3}_{\mathbb{Q}}.
 \] The proof of Theorem \ref{BIGTHM} shows that for all but finitely $n\in \mathbb{Z}$ we have $\Br U_{n}=\Br X_{n}=\Br \mathbb{Q}$, where $X_{n}$ is the compactification of $U_{n}$  One should then expect the surfaces $U_{n}$ to rarely fail weak approximation. Consider the case where the compactificaiton $X_{n}$ is smooth and $\Br X_{n}=\Br \mathbb{Q}$. Note that $X_{n}(\mathbb{Q})\neq \emptyset$ then as $X_{n}$ is a smooth cubic surface the set $X(\mathbb{Q})$ will be Zariski dense in $X_{n}$, this implies $U_{n}(\mathbb{Q})\neq \emptyset$. Under a conjecture of Colliot-Thélène \cite[Conj 14.1.2]{CTS21}, $X(\mathbb{Q})$ is dense in $X_{n}(\mathbb{A}_{\mathbb{Q}})^{\Br}=X_{n}(\mathbb{A}_{\mathbb{Q}})$ i.e. if this conjecture holds $X_{n}$ satisfies weak approximation. As $X_{n}$ and $U_{n}$ are smooth birational varieties which are everywhere locally soluble, under conjecture \cite[Conj 14.1.2]{CTS21}, $U_{n}$ would also satisfy weak approximation. In particular, the example of sum of three tetrahedral numbers given in Section \ref{SumOfThreeTetSection} will satisfy weak approximation if conjecture \cite[Conj 14.1.2]{CTS21} is true.
 \end{remark}

\subsection{Example}
\begin{proposition}[{\cite[Prop 13.3.12]{CTS21}}]\label{prop: WA and Brauer group}
Let $X$ be a projective variety over $k$ such that $X(k)\neq \emptyset$. If exists $\alpha\in \Br X$ where $\alpha$ takes at least two different values on $X(k_{v})$ for a place $v\in \Omega_{k}$, then $X$ fails weak approximation.
\end{proposition}

 \begin{theorem}[{\cite[Cor 11.3.5]{CTS21}}] \label{BigThm}
 Let $X \rightarrow \mathbb{P}^{1}_{k}$ be a conic bundle. If we fix $m \in \mathbb{P}^{1}_{k}$ a closed point such that $\pi^{-1}(m)$ is smooth. Let $S\subset \mathbb{P}^{1}_{k}$ be the the finite set of closed points with singular fibres. We define $\mathbb{P}^{1} \supset \mathbb{A}^{1}_{k}=\mathbb{P}^{1}_{k}\backslash \{m\}$ with coordinates $t$. Then 
 \begin{align*}
     \Br X/\Br k=\pi^{*}(B)
 \end{align*}where $B \subset \Br(k(\mathbb{P}^{1}))=\Br(k(t))$ is a finite subgroup. Moreover, we have an explicit description of the elements in $B$. A closed point $p \in S$ is the zero set of some monic irreducible polynomial $P(t)\in k[t]$. Let $k(p)=k[t]/(P(t))$ be the residue field of $p$ and denote $\tau_{p}$ to be the image of $t$. Then there exists a quadratic extension $F_{p}=k(p)(\sqrt{a_{p}})$ such that the fibre above $p$ splits over $F_{p}$ into two transversal lines, now we can define the set $W \subset \mathbb{F}_{2}^{|S|}$ of vectors $\varepsilon = (\varepsilon_{p})_{p \in S}$ such that
$$ 
     \prod\limits_{p \in S}\left(\Norm_{k(p)/k}(a_{p})\right)^{\varepsilon_{p}}=1 \in k^{*}/(k^{*})^{2}.
$$ Then there is an injective map $W\rightarrow \Br k(t)$ which sends $\varepsilon$ to
$$
     A_{\varepsilon}= \sum\limits_{p \in S}\varepsilon_{p} \Cores_{k(p)/k}(t-\tau_{p},a_{p})
$$ where $(t-\tau_{p},a_{p})$ represents a quaternion algebra in $\Br(k(p)(t))$.
 \end{theorem}

 \begin{example} We now present an example which fails both strong and weak approximation, namely the affine suface \[
U_{50}:u_{1}^3+u_{2}^3+u_{3}^3-15(u_{1}+u_{2}+u_{3})+50=0.
 \]
Consider the cubic surface 
\[
    X_{50}:x_{1}^3+x_{2}^3+x_{3}^3-15(x_{1}+x_{2}+x_{3})x_{0}^2+50x_{0}^3=0.
\]
Then $X_{50}$ contains a line, namely $L=\{x_{1}+x_{2}=0,x_{3}+5x_{0}=0\}$, hence $X_{50}$ admits a conic bundle structure. We apply a change of variables $x_{1}+x_{2}\mapsto x_{1}, x_{3}+5x_{0}\mapsto x_{3}$ so we get a smooth cubic surface $X_{50}^{'}$ which is isomorphic to $X_{50}$ over $\mathbb{Q}$ and where $L':=\{x_{1}=x_{3}=0\}$ is the image of $L$ under this isomorphism. Using the method described in Remark \ref{ConstructConicBundle} we can construct the conic bundle map associated to the line $L'$, $\pi:X'_{50}\rightarrow \mathbb{P}^{1}_{\mathbb{Q}}$. The closed points of $\mathbb{P}^{1}_{\mathbb{Q}}$ which correspond to singular fibres of $\pi$ are $\{(t),(s-t),(s+t),(s^2 - 4st + t^2)\}$. We now use Theorem \ref{BigThm} to construct elements of $\Br X'_{50} /\Br \mathbb{Q}$.
\begin{enumerate}
    \item The residue field of the point $[0:1]$ is $\mathbb{Q}$ and the fibre above this point is $60x_0^2 - 15x_0x_1 + x_1^2=0$ which splits over $\mathbb{Q}(\sqrt{-15})$,
    \item The residue field of the point $[1:-1]$ is $\mathbb{Q}$ and the fibre above this point is split over $\mathbb{Q}$,
    \item The residue field of the point $[1:1]$ is $\mathbb{Q}$ and the fibre above this point is $45x_0^2 - 15x_0x_1 + 2x_1^2 - 3x_1x_2 + 3x_2^2=0$ which splits over $\mathbb{Q}(\sqrt{-15})$,
    \item The residue field of the point $p=(s^2-4st+t^2)$ is $\mathbb{Q}(\sqrt{12})$. The fibres above the two $\mathbb{Q}(\sqrt{12})$-points are split over $\mathbb{Q}(\sqrt{-15})$.
\end{enumerate}
We now can apply Theorem \ref{BigThm} and can see a non-trivial choice for $\varepsilon$ is:\[
    \varepsilon_{p} =
\begin{cases}
1 \ \text{for}\ p=[0:1],\\
0 \ \ \text{for}\ p=[1:-1],\\
1  \ \text{for}\ p=[1:1],\\
0  \ \text{for}\ p=(s^2-4st+t^2).\\
\end{cases}
\]
Then 
\[
    \prod\limits_{p \in S}\left(\Norm_{k(p)/k}(a_{p})\right)^{\varepsilon_{p}}= (-15)^2=1\  \left(\text{mod}\  (\mathbb{Q}^{\times})^2\right).
\] Using Magma we can check that $\Br X^{'}_{50}/\Br \mathbb{Q}\cong \mathbb{Z}/2\mathbb{Z}$, hence we can define the generator for $\Br X'_{50} /\Br \mathbb{Q}$ specifically the pullback of  
$A_{\varepsilon}=(x(x-1),-15)$ where $x=x_{1}/x_{3}$, this substitution acts as the pullback map i.e.
\[
    \Br X^{'}_{50}/\Br \mathbb{Q} \  \text{is generated by} \ \alpha' = \left(\left(\frac{x_{1}}{x_{3}}-1\right)\frac{x_{1}}{x_{3}},-15\right).
\] We can then deduce \[
    \Br X_{50}/\Br \mathbb{Q} \  \text{is generated by} \ \alpha = \left(\left(\frac{x_{1}+x_{2}}{x_{3}+5x_{0}}-1\right)\frac{x_{1}+x_{2}}{x_{3}+5x_{0}},-15\right).
\]
Let $p=\infty$ and $\boldsymbol{x}_{1}=[-1/10 :1/2 :-1:1]$ and $\boldsymbol{x}_{2}=[1/2 : 7/6 : 11/6 :1]$ on $X_{50}(\mathbb{R})$, and we shall denote $\alpha(\boldsymbol{x}_{i})$ for the image of $\boldsymbol{x}_{i}$ under the evaluation map $\text{ev}_{\alpha,\infty}$. Then we have 
\[
  \text{inv}_{\infty}(\alpha(\boldsymbol{x}_{i})) =
    \begin{cases}
      1 \ &\text{if} \ \alpha(\boldsymbol{x}_{i}) \ \text{is split over} \ \mathbb{R}\\
      -1 \ &\text{if} \ \alpha(\boldsymbol{x}_{i}) \ \text{is not split over} \ \mathbb{R}.
    \end{cases}       
\]
 We see that $\alpha(\boldsymbol{x}_{1})=(2,-15)$ which is split over $\mathbb{R}$ and $\alpha(\boldsymbol{x}_{2})=(-6/49,-15)$ which is not split over $\mathbb{R}$, hence $\text{inv}_{\infty}$ takes both values. As $X_{50}$ is a smooth projective variety such that $X_{50}(\mathbb{Q})\neq \emptyset$, Proposition \ref{prop: WA and Brauer group} tells us that $X_{50}$ fails weak approximation. It is clear that $U_{50}(\mathbb{Q})\neq \emptyset$ and $U_{50}$ is birational to $X_{50}$, hence by Lemma \ref{lemma: WA birational invariant} the failure of weak approximation on $X_{50}$ implies that $U_{50}$ fails weak approximation. Finally, as strong approximation implies weak approximation $U_{50}$ must also fail strong approximation.
 \end{example}

\section{Tables}
Table \ref{PossGalType} describes the possible Galois actions on the surface $\bar{X}_{n}$ with splitting field $k$ over $\mathbb{Q}$. The first column describes $\Gal(k/\mathbb{Q})$, the second column describes the possible Galois action on the lines i.e. $[1^3,2^{12}]$ means there are 3 orbits of size 1 and 12 orbits of size 2. Then the final column describes the first Galois cohomology of $\Pic \bar{X}_{n}$.
\begin{table}[H]
\caption{\label{PossGalType} Possible Galois Types}
\begin{tabular}{|l|l|l|}\hline
Galois Group & Orbit Type & $\HH^1(\mathbb{Q},\Pic \bar{X}_{n})$ \\\hline
$C_1, C_{2}$ & $[1^{27}],[1^{15}, 2^{6}]$ & $0$ \\\hline
$C_2$ & $[1^{3}, 2^{12}]$ & $(\mathbb{Z}/2\mathbb{Z})^2$ \\\hline
$C_2, C_{3}, C_{3}$ & $[1^{3}, 2^{12}],[1^{9}, 3^6],[3^9]$ & $0$ \\\hline
$C_3$ & $[3^{9}]$ & $(\mathbb{Z}/3\mathbb{Z})^2$ \\\hline
$C_2^2$ & $[1^{3}, 2^{6}, 4^{3}]$ & $\mathbb{Z}/2\mathbb{Z}$ \\\hline
$S_3$ & $[3^{3}, 6^{3}]$ & $(\mathbb{Z}/2\mathbb{Z})^2$ \\\hline
$S_3, C_{6}, C_{6}, S_3, S_{3}$ & $[1^{9}, 3^{6}], [1^{3}, 2^{3}, 6^{3}], [3^{5}, 6^{2}], [1^{3}, 2^{3}, 6^{3}], [3^{3}, 6^{3}]$ & $0$ \\\hline
$S_3$ & $[3^{3}, 6^{3}]$ & $\mathbb{Z}/3\mathbb{Z}$ \\\hline
$C_3^2$ & $[3^{3}, 9^{2}]$ & $\mathbb{Z}/3\mathbb{Z}$ \\\hline
$C_{2}\times S_{3}$ & $[3^{3}, 6, 12]$ & $\mathbb{Z}/2\mathbb{Z}$ \\\hline
$C_{2}\times S_{3}, C_3\rtimes S_3$ & $[1^{3}, 2^{3}, 6^{3}],[3^{3}, 18]$ & $0$ \\\hline
$C_3\times S_3$ & $[3^{3}, 9^{2}]$ & $\mathbb{Z}/3\mathbb{Z}$ \\\hline
$C_3\times S_3, S_{3}^{2}$ & $[3^{3}, 18],[3^{3}, 18]$ & $0$ \\\hline
\end{tabular}
\end{table}
\bibliographystyle{amsalpha}{}
\bibliography{bibliography/referencesproject0}
\end{document}